\documentclass[11pt,a4paper,usenames,dvipsnames,reqno]{amsart}

\usepackage{amssymb,amsfonts}
\usepackage{mathtools}
\usepackage[margin=1.25in]{geometry}
\usepackage[all,arc]{xy}
\usepackage{enumitem}
\usepackage{mathrsfs}
\usepackage{color}   
\usepackage{hyperref}
\hypersetup{
    colorlinks=true, 
    linktoc=all,     
    linkcolor=blue,  
}
\usepackage{amsthm}
\usepackage{amsmath}
\usepackage{graphicx}
\usepackage{wasysym}
\usepackage{pgf,tikz}
\usetikzlibrary{positioning}

\usepackage{phaistos}
\usepackage{graphicx}   
\usepackage{epsfig}     
\usepackage{endnotes}
\usepackage{imakeidx}
\usepackage{amscd}
\usepackage{tikz-cd}
\usepackage{bbm}

\newcommand{\cc}{\mathbb C}
\newcommand{\ff}{\mathbb F}

\newcommand{\zz}{\mathbb Z}

\newcommand{\qq}{\mathbb Q}

\newcommand{\la}{\langle}
\newcommand{\ra}{\rangle}
\newcommand{\lra}{\longrightarrow}
\newcommand{\hra}{\hookrightarrow}

\newcommand{\ep}{\epsilon}
\newcommand{\lam}{\lambda}

\DeclareMathOperator{\End}{End}
\DeclareMathOperator{\Ext}{Ext}

\DeclareMathOperator{\Hom}{Hom}

\DeclareMathOperator{\Mod}{Mod}

\DeclareMathOperator{\supp}{supp}
\DeclareMathOperator{\Perv}{Perv}

\makeatletter
\newcommand{\colim@}[2]{%
  \vtop{\m@th\ialign{##\cr
    \hfil$#1\operator@font colim$\hfil\cr
    \noalign{\nointerlineskip\kern1.5\ex@}#2\cr
    \noalign{\nointerlineskip\kern-\ex@}\cr}}%
}
\newcommand{\colim}{%
  \mathop{\mathpalette\colim@{\rightarrowfill@\textstyle}}\nmlimits@
}
\makeatother

\newcommand{\ft}{\mathfrak t}

\newcommand{\calGr}{\mathcal{G}r}

\newcommand{\cE}{\mathcal E}
\newcommand{\cF}{\mathcal F}
\newcommand{\cC}{\mathcal C}
\newcommand{\cB}{\mathcal B}
\newcommand{\cA}{\mathcal A}
\newcommand{\cP}{\mathcal P}
\newcommand{\cI}{\mathcal I}

\makeatletter
\def\Ddots{\mathinner{\mkern1mu\raise\p@
\vbox{\kern7\p@\hbox{.}}\mkern2mu
\raise4\p@\hbox{.}\mkern2mu\raise7\p@\hbox{.}\mkern1mu}}
\makeatother

\newcommand{\vp}{\varpi}
\newcommand{\T}{\mathcal{T}}
\newcommand{\F}{\mathcal{F}}

\newcommand{\G}{\mathcal{G}}

\newcommand{\bF}{\mathbb{F}}
\newcommand{\bO}{\mathbb{O}}

\newcommand{\bT}{\mathbb{T}}
\newcommand{\Perf}{\mathrm{Perf}}
\newcommand{\kvp}{\mathbb{F}[\varpi]}
\newcommand{\zvp}{\mathbb{O}[\varpi]}
\newcommand{\Psm}{\mathbf{Psm}}

\newcommand{\liee}{T^0}
\newcommand{\liet}{\mathfrak{t}}
\newcommand{\lieo}{T^1}
\newcommand{\modc}{\text{-}\mathrm{Mod}_c}
\newcommand{\moddg}{\text{-}\mathrm{Mod}_{dg,c}}

\swapnumbers
 \theoremstyle{plain}
  \newtheorem{Thm}[subsubsection]{Theorem}
  \newtheorem{Lem}[subsubsection]{Lemma}
  \newtheorem{Prop}[subsubsection]{Proposition}
  \newtheorem{Cor}[subsubsection]{Corollary}

 \theoremstyle{definition}
  \newtheorem{Def}[subsubsection]{Definition}
  \newtheorem{Rem}[subsubsection]{Remark}
  \newtheorem{Cond}[subsubsection]{Condition}

\title[Parity sheaves and Smith theory]{Parity sheaves and Smith theory}

\author{Spencer Leslie}
\author{Gus Lonergan}

\theoremstyle{plain}

\theoremstyle{definition}

\begin{document}

\maketitle

\begin{abstract}
Let $p$ be a prime number and let $X$ be a complex algebraic variety with an action of $\mathbb{Z}/p\mathbb{Z}$. We develop the theory of parity complexes in a certain $2$-periodic localization of the equivariant constructible derived category $D^b_{\mathbb{Z}/p\mathbb{Z}}(X,\mathbb{Z}_p)$. Under certain assumptions, we use this to define a functor from the category of parity sheaves on $X$ to the category of parity sheaves on the fixed-point locus $X^{\mathbb{Z}/p\mathbb{Z}}$. This may be thought of as a categorification of Smith theory. When $X$ is the affine Grassmannian associated to some complex reductive group, our functor gives a geometric construction  of the Frobenius-contraction functor recently defined by M. Gros and M. Kaneda via the geometric Satake equivalence.
\end{abstract}

\setcounter{tocdepth}{1}
\tableofcontents

\section{Introduction}\label{sec:intro}

Let $\bF$ be a field of characteristic $p>2$, and let $X$ be a complex algebraic variety equipped with an action of the cyclic group $\varpi$ of order $p$. Broadly speaking, when working with $\bF$-coefficients, homological data of $X$ should be related to homological data of the fixed point locus $X^{\varpi}$. This idea goes back at least to P. Smith's work \cite{Sm}, who related the vanishing of $H^*(X,\bF)$ to that of $H^*(X^{\varpi},\bF)$. One may ask whether a similar relationship exists between more modern (e.g. categorical) data associated to $X$. The main aim of this paper is to relate the category of \emph{parity sheaves} on $X$ with that on $X^{\varpi}$.

Let $\bO=W(\ff)$ be the ring of Witt vectors for the field $\ff$; for example, if $\ff=\ff_p$ then $\bO=\zz_p$. In the most general setting, we construct a functor $LL$ (see \ref{liftingsection}) of the form
\[
\left\{\parbox{15em}{\centering$\vp$-equivariant parity sheaves on $X$\\ with coefficients in $\bO$}\right\}\xrightarrow{LL} \bigg\{\parbox{16em}{\centering parity sheaves on the fixed points $X^\vp$\\ with coefficients in $\ff$.}\bigg\};
\]permit us to postpone the precise descriptions of the categories of parity sheaves until (\ref{lets get technical}) as they are rather technical in this generality. 

The construction of $LL$ rests on our development of a theory of parity sheaves with coefficients in a certain $\mathbb{E}_\infty$-ring spectrum $\T_\infty$, though one does not need to understand the theory of $\mathbb{E}_\infty$-rings to describe our results. In fact, our approach is to study these sheaves as objects in a certain localization of a derived category. This is an integral version of D. Treumann's ``Smith theory for sheaves'' (\cite{Tr}). We show that there is a good notion of parity object in this setting which interpolates between parity sheaves on $X$ and $X^\vp.$ A key innovation is our construction of functors between \emph{classical} categories of parity sheaves using this Smith theory; see the discussion below Theorem \ref{Thm: Introduction} below.

While our approach is topological in nature, our primary motivation is applications to geometric representation theory, where parity sheaves play an important role in the modular setting (e.g. \cite{JMW1}, \cite{JMW2}, \cite{Williamsonexpos}, \cite{MR3481352}, and the references therein). Our main example is the affine Grassmannian
\[
\calGr_G=G(\cc((t))/G(\cc[[t]])
\]of a complex reductive algebraic group $G$ with its spherical stratification; here we use the standard convention and consider $\calGr_G$ with its analytic topology. This space may be equipped with the so-call \emph{loop-rotation action} of $\vp$, generated by the shift of the uniformizer
\[
t\longmapsto \zeta_pt
\]
for a primitive $p$-th root of unity $\zeta_p\in \cc^\times$. One of the components of the fixed points $\calGr_G^\vp$ is naturally isomorphic to $\calGr_G$ (namely, elements of $G(\cc((t^p))/G(\cc[[t^p]])$), and we restrict our attention to this component. Specializing our constructions to this setting, we obtain a more refined functor
\begin{align*}
\underline{\Psm}:\mathrm{Par}^0_{sph}(\calGr_G;\ff)\to \mathrm{Par}^0_{sph}(\calGr_G;\ff),
\end{align*}
where $\mathrm{Par}^0_{sph}(\calGr_G;\ff)$ is the category of spherical perverse parity complexes over $\ff$. For $p$ good for $G$ \cite{MR}, this category is equivalent under the geometric Satake equivalence to the category of tilting modules for the Langlands dual algebraic group $G^\vee_{\ff}$ over $\ff$. The following theorem is our main application.
\begin{Thm}\label{Thm: intro motivation}
There is an extension of $\underline{\Psm}$ to the full category of spherical perverse sheaves
\begin{align*}
\underline{\Psm}:\mathrm{Perv}_{sph}(\calGr_G;\ff)\to \mathrm{Perv}_{sph}(\calGr_G;\ff)
\end{align*}
such that, under the geometric Satake equivalence, the resulting functor
\begin{align*}
\mathrm{Rep}(G^\vee_{\ff})\to \mathrm{Rep}(G^\vee_{\ff})
\end{align*}
gives the Frobenius contraction functor of \cite{GK0}.
\end{Thm}
See Section \ref{Section: geometric application} for a more complete discussion on the geometric Satake equivalence and the Frobenius contraction functor. In particular, Smith theory gives a geometric construction of the Frobenius contraction functor. We remark that recently Riche and Williamson have further studied the Smith theory of the loop rotation in connection with the other components of the fixed-points $\calGr_G$, with applications to the linkage principle \cite{RicheWilliamsonSmith}.

The rest of the introduction outlines our approach.

\subsection{Smith theory for sheaves} Let $\ff[\vp]$ denote the group algebra and let $D^b_c(\bF[\varpi])$ denote the (cohomologically bounded, with finitely generated cohomology groups) derived category of the category of $\ff[\vp]$-modules. Write $\mathrm{Perf}_c(\bF[\varpi])$ for the thick subcategory spanned by bounded complexes of free $\kvp$-modules. The \emph{Tate category} is defined to be the Verdier quotient

\[
\T_0\modc :=D^b_c(\kvp)/\mathrm{Perf}_c(\kvp).
\]
As pointed out in \cite[Section 4.1]{Tr}, this localization (denoted as $\Perf(\T)$ in \emph{loc. cit.}) is equivalent to the homotopy category of the category of finitely-generated module spectra over a certain $\mathbb{E}_\infty$-ring spectrum $\T_0$. This point of view is not technically necessary for this paper, but is philosophically useful and hopefully justifies our notation.

Treumann also defines relative Tate categories of sheaves valued in $\T_0\modc$. More precisely, given an algebraic variety $Y$ with trivial $\vp$-action, one may similarly consider the relative Tate category given as the quotient
 \[
\mathrm{Sh}_c(Y;\T_0) :=D^b_c(Y;\kvp)/\mathrm{Perf}_c(Y;\kvp)
\]
of the (cohomologically-bounded, algebraically constructible) equivariant derived category 
\[
D^b_{\vp,c}(Y;\bF)\cong D^b_c(Y;\kvp)
\]
by the full subcategory generated by sheaves whose stalks vanish in $\T_0\modc$. Again, while not strictly necessary, it is fruitful to view (as in \cite{Tr}) this localization as a subcategory of the homotopy category of sheaves of module spectra over $\T_0$, which motivates our notation.  These quotients are not themselves derived categories; for instance the double suspension $[2]$ is isomorphic to the identity, giving a $2$-periodic structure to $\mathrm{Sh}_c(Y;\T_0)$.

For a $\vp$-variety $X$ with fixed-point subvariety $Y=X^\vp$, one obtains the sheaf-theoretic \emph{Smith functor}
\[
\Psm : D^b_{\vp,c}(X;\bF) \lra \mathrm{Sh}_c(X^\vp; \T_0)
\]
by $*$-restricting to $X^\vp$ and then projecting. This operation should be regarded as the $\vp=\mu_p$ analogue of hyperbolic localization (see e.g. \cite{Br}), relating mod $p$ topological data of $X$ to that of $X^\vp$. In particular, Treumann conjectures that $\Psm$ is well behaved when restricted to the subcategory of perverse sheaves on $X$ (see \cite[Conjecture 4.1]{Tr}, which the author establishes in several special cases).

Let us fix a $\vp$-equivariant stratification $S$ of $X$, which induces in the natural manner a stratification also denoted by $S$ of $X^\vp$. It is shown in \cite{Tr} that the Smith functor $\Psm$ preserves constructibility with respect to $S$ and commutes with all the standard operations on sheaves. 

\begin{Rem}
For simplicity, we will assume throughout that the stratification on $X^\vp$ consists only of \textbf{simply-connected strata}. This is not strictly necessary for much of the theory, but our results are most complete in this setting. In any case, this assumption is satisfied for many cases of interest to geometric representation theory.
\end{Rem}

We observe that the $2$-periodicity $[2]\cong Id$ of $\mathrm{Sh}_S(X^\vp;\T_0)$ makes it possible to define a notion of \emph{parity complexes} in analogy to the theory developed in \cite{JMW1}; we may then ask if $\Psm$ takes parity objects to parity objects. In cases when parity sheaves on $X$ and $X^\vp$ are known to be perverse, we can hope to use this to relate perverse sheaves on $X$ to those on $X^\vp$ in accordance with \cite[Conjecture 4.1]{Tr}.

Indeed, one may define a $\mathbb{Z}/2\zz$-graded cohomology functor on $\mathrm{Sh}_S(X^\vp; \T_0)$, and define \emph{Tate-parity complexes} in terms of this functor. However, this definition has the undesirable property that the projection to the Tate category of indecomposible parity sheaves on $X^\vp$ need not be indecomposable; in fact, even over a point, the projection of the module $\ff$ will be the sum of a non-zero even and non-zero odd part (see Remark \ref{Rem: modular coefficients}). This is the Bockstein phenomenon, and our remedy is to replace $\bF$ by its Witt ring $\bO$.

This leads us to extend the construction of \cite{Tr} to sheaves of $\bO$-modules. The resulting theory retains all of the desirable features of the modular variant; in particular, we have a localized category $\mathrm{Sh}_S(X^\vp;\T_\infty)$ 
and an integral Smith functor 
\[
\Psm : D^b_{\vp,S}(X;\bO) \lra \mathrm{Sh}_S(X^\vp; \T_\infty).
\]
Importantly, this integral setting has a better-behaved theory of parity objects. Even the structure of the coefficient category 
\[
\T_\infty\modc:= D^b_c(\bO[\varpi])/\mathrm{Perf}_c(\bO[\varpi]),
\]
where $\mathrm{Perf}_c(\bO[\varpi])$ is the thick subcategory of \emph{weakly-injective} modules (see Section \ref{sec:const}), is more transparent: Corollary \ref{Prop: Grothendieck} states it is equivalent to $\zz/2\zz$-graded finite dimensional $\ff$-vector spaces (see Remark \ref{Rem: why not two}). Correspondingly, we obtain a well-behaved relative theory of parity objects as well. More specifically, in Section \ref{Section: sheaves} we construct a functor of $\zz/2\zz$-graded cohomology sheaves of $\ff$-vector spaces on $\mathrm{Sh}_S(X^\vp;\T_\infty)$ and define a notion of \emph{Tate-parity complex} in $\mathrm{Sh}_S(X^\vp;\T_\infty)$ in terms of this functor (see Definition \ref{Def: Tate parity}). The resulting theory enjoys many of the properties of the classical theory of parity sheaves, including a complete classification of the indecomposible objects.

Before stating our first main result, let us fix a little more notation: consider the counit
\begin{align*}
\ep: \bO[\vp]&\lra\bO\\
        g&\longmapsto 1,
\end{align*}
where $g$ denotes a generator of $\vp$; this induces a functor
\begin{align*}
 \epsilon^\ast:D^b_S(X^\vp;\bO)\to D^b_S(X^\vp;\bO[\vp]),
\end{align*}
defined by endowing $\bO$-modules with the trivial $\vp$-action. We also set
\begin{align*}
\bT^*:D^b_S(X^\vp;\bO[\vp])\lra \mathrm{Sh}_S(X^\vp;\T_\infty)
\end{align*}
for the defining projection functor; set $\bT:=\bT^* \epsilon^\ast.$ We call this the \emph{Tate extension of scalars} functor.
Our first main theorem is the following classification result (see Sections \ref{Section: tateparitysheaves} and \ref{Section: JWM redux}):

\begin{Thm}\label{Thm: good theory intro} Let $S$ be a stratification of $X^\vp$ satisfying the JMW condition (\ref{parcon}) and consisting of simply-connected strata. Fix a pariversity $\dagger:S\lra \zz/2\zz$ on $X^\vp$. 
\begin{enumerate}\item\label{gus1} The functor
\begin{align*}
\bT:D^b_S(X^\vp;\bO)\lra \mathrm{Sh}_S(X^\vp; \T_\infty)
\end{align*}
sends $\dagger$-parity complexes to $\dagger$-Tate-parity complexes.
\item\label{gus2} Up to shifting by $[1]$, there is at most one indecomposable $\dagger$-Tate-parity complex in $\mathrm{Sh}_S(X^\vp;\T_\infty)$ supported on the closure of a given stratum.
\item\label{gus3} Let $\mathcal{E}$ be an indecomposable $\dagger$-parity complex in $D^b_S(X^\vp;\bO)$ whose negative self
-extensions all vanish (this is satisfied if $\mathcal{E}$ is perverse, for example). Then $\bT\mathcal{E}$ is indecomposable.
\end{enumerate}
\end{Thm}
Following \cite{JMW1}, we call a parity complex (resp., Tate-parity complex) a \emph{parity sheaf} (resp., a \emph{Tate-parity sheaf}) if it is indecomposable and its restriction to any stratum which is dense in its support is supported in the appropriate cohomological degree. The theorem implies that if $D^b_S(X^\vp;\bO)$ has `enough' parity sheaves, $\bT$ induces a bijection
\[
\bT:\left\{\text{parity sheaves on $X^\vp$}\right\}\lra \left\{\text{Tate-parity sheaves on $X^\vp$}\right\}.
\] 
\begin{Rem}
The simple-connectedness assumption only enters in \ref{gus3} and is probably unnecessary. Nevertheless, our proof relies on it.
\end{Rem}
This result mirrors the structure theory of parity sheaves closely.  The arguments rely centrally on Theorem \ref{megathm}, which classifies Tate-parity sheaves in the smooth case (say, on a stratum $Y_\lam$). With this in hand, we show in Proposition \ref{prop: krull} that Tate-parity complexes decompose uniquely into sums of indecomposible objects, at which point the proof of Theorem \ref{Thm: good theory intro} is remarkably similar to the arguments in \cite[Section 2]{JMW1}. The proofs of Section \ref{Section: Tate-parity} rely on our structural results about the category $\mathrm{Sh}_S(X^\vp;\T_\infty)$ and certain technical tools (e.g. the {Tate hypercohomology spectral sequence}; see \ref{hyperco}) introduced in Section \ref{Section: sheaves}.

\subsection{Smith theory and parity sheaves}\label{lets get technical} With a theory of Tate-parity sheaves in hand, the Smith functor
\begin{align*}
\Psm:D^b_{\vp,S}(X;\bO)\lra \mathrm{Sh}_S(X^{\vp};\T_\infty)
\end{align*}
may be used to relate parity sheaves on $X$ to those on $X^\vp$, as we now explain. First, under certain additional hypotheses we show that $\Psm$ 
sends $\vp$-equivariant parity complexes on $X$ to Tate-parity complexes on $X^\vp$. Here, a $\vp$-equivariant parity complex simply means an object of $D^b_{\vp,S}(X;\bO)$ whose underlying complex in $D^b_S(X;\bO)$ is parity. A special case of our second main result (stated in Theorem \ref{alloallo}; see Lemma \ref{s1lem}) is the following theorem:

\begin{Thm}\label{Thm: Introduction}
Suppose the $\vp$-action on $X$ extends to an action of $S^1$ such that every connected component of every stratum of $X^\vp$ contains an $S^1$-fixed point. For any  parity complex $\cE$ in $D_{\vp,S}^b(X;\bO)$ that comes from an  $S^1$-equivariant sheaf, $\Psm(\cE)$ is Tate-parity.
\end{Thm}
In particular, we obtain a second functor valued in the the additive category of Tate-parity complexes. That is, supposing the assumptions of Theorems \ref{Thm: good theory intro} and \ref{Thm: Introduction} hold (for example, the loop rotation on $X=\calGr_G$), these theorems imply that for a parity complex $\cE$ in $D^b_{\vp,S}(X;\bO)$, there is a unique parity complex $\cF$ in $D^b_{S}(X^\vp;\bO)$ satisfying
\begin{equation}\label{corresp}
\Psm(\cE)\cong \bT\cF.
\end{equation}
 We show in Section \ref{liftingsection} that this correspondence gives rise to a \emph{functor}
\begin{align*}
    LL:\mathrm{Par}^n_{\vp,S}(X;\bO)&\lra \mathrm{Par}_{S}^0(X^\vp;\bF),\\
    \mathcal{E}\qquad&\longmapsto \quad \ff\F
\end{align*}
where the superscripts indicate restricting to certain subcategories (related to e.g. codimensions of stata; see Section \ref{liftingsection} for full details). Remarkably, this functor between categories of (classical) parity complexes is constructed by passing through the integral Tate category.  We view this as a sort of refined $\vp$-localization functor, defined only on categories of parity sheaves. 

The existence of this functor relies on Theorem \ref{lifting}, showing that the modular reduction functor 
\[
\ff(-):D_c^b(X^\vp;\bO)\lra D_c^b(X^\vp;\bF).
\]
factors through the Tate extension of scalars functor $\bT$. We prove this by exploiting an additional graded structure on 
\[
\Hom_{\mathrm{Sh}_c(X^\vp;\T_\infty)}(\bT\F,\bT\G),
\]
the degree zero factor of which is naturally isomorphic to $\Hom_{D^b_c(X^\vp;\ff)}(\ff\F,\ff\G),$; see Section \ref{addendum}.
\begin{Rem}
One of our initial motivations was to establish examples of Treumann's perversity conjecture \cite[Conjecture 4.1]{Tr}, which our results accomplish for perverse parity sheaves. As we see below, for more general applications an important point is that functors on perverse sheaves are often determined by their action on perverse parity sheaves.
\end{Rem}

\subsection{Application to geometric representation theory} Now fix a connected reductive group $G$ over $\cc$, and consider $X=\calGr_G$ equipped with the loop rotation. Assuming $p$ is good for $G$, we apply the above results to prove Theorem \ref{Thm: intro motivation} in Section \ref{Section: geometric application}. 

Two aspects of this proof warrant further comment. First, in this case we show that the localization functor $LL$ factors further\footnote{More precisely, the composition of $LL$ with certain equivalences factors in this way; see Section \ref{Section: Tate perverse}.}
\[
\begin{tikzcd}
\mathrm{Par}^0_{sph}(\calGr_G;\bO)\ar[d,"\ff"]\ar[rd,"LL"]&\\
\mathrm{Par}^0_{sph}(\calGr_G;\ff)\ar[r,"\underline{\Psm}"]&\mathrm{Par}^0_{sph}(\calGr_G;\ff),
\end{tikzcd}
\]
where $\ff$ is the  modular reduction functor. The resulting functor $\underline{\Psm}$ is the appropriate functor to consider in the context of the geometric Satake equivalence. Second, by combining Treumann's modular Smith functor and a functor we refer to as \emph{Tate-perverse cohomology} we construct a functor
\[
\Perv_{sph}(\calGr_G;\ff)\lra \Perv_{sph}(\calGr_G;\ff)
\]
in Section \ref{Section: Tate perverse}. It is this functor that appears in the statement of Theorem \ref{Thm: intro motivation}, though it is difficult to say much about it directly.

Crucially, Proposition \ref{Prop: smith extension} shows that the restriction of this functor to perverse parity sheaves is precisely $\underline{\Psm}$. Using this result, we first establish Theorem \ref{Thm: intro motivation} directly on the subcategory $\mathrm{Par}^0_{sph}(\calGr_G;\ff)$ and then show this is enough to conclude the full result via an abstract argument.

\subsection{Outline}
After we fix notation and conventions in Section \ref{Section: prelim}, we study the Tate coefficient category in Section \ref{sec:const}. Turning to the relative theory, Section \ref{Section: sheaves} develops our theory of sheaves valued in $\T_\infty$. In particular, we introduce our notion of Tate cohomology sheaves and establish several structural results about $\mathrm{Sh}_c(Y;\T_\infty).$ In Section \ref{Section: Tate-parity}, we recall the notion of parity sheaves from \cite{JMW1} and develop the theory of Tate-parity sheaves. Section \ref{Section: Smith} establishes our main results regarding the interaction between parity sheaves and the Smith functor. Finally, we review the geometric Satake equivalence and the  Frobenius contraction functor and prove Theorem \ref{Thm: intro motivation} in Section \ref{Section: geometric application}.
\subsection{Acknowledgements}
{ We wish to thank Boston College and the Massachusetts Institute of Technology where the authors carried out the bulk of this work as graduate students. We wish to thank Daniel Juteau, Pramod Achar, David Treumann, and Geordie Williamson for helpful discussions on topics related to this work. S.L. thanks David Treumann for conversations leading to this project and for providing travel funding while a graduate student. He also wishes to thank his advisor Solomon Friedberg for all his help and support. G.L. wishes particularly to thank Pramod Achar for bringing the work of David Treumann to his attention, and his advisor Roman Bezrukavnikov for all his support and encouragement. Finally, we thank the anonymous referee for several suggestions and corrections which dramatically improved the paper.}

\section{Preliminaries}\label{Section: prelim}
\subsection{Notation}
Let $p$ be an odd prime and let $\bF$ denote an algebraic extension of the finite field $\ff_p,$ and let $\bO$ denote the ring of Witt vectors. For example, if $\bF=\bF_{p^n}$ then $\bO=\zz_{p^n}$ is the ring of integers of the unique unramified degree $n$ extension $\qq_{p^n}$ of $\qq_p.$

Now let $\vp$ denote a cyclic group of order $p$. We fix a generator $g\in \vp$ and set 
\[
N=1+g+\cdots g^{p-1}
\]
for the norm element of the group ring $\bO[\vp]$.

\subsection{Conventions on sheaves}
 Let $X$ be a complex algebraic variety endowed with the analytic topology. Throughout the paper, a \emph{stratification} of $X$ will mean an algebraic Whitney stratification. It is well known that every algebraic but not necessarily Whitney stratification can be refined to an algebraic Whitney stratification. All of our arguments work equally well with a fixed algebraic Whitney stratification as without any fixed choice. 
 
 For any category whose objects are complexes of sheaves on $X$ (or $X^\vp$) with constructible cohomology sheaves, we will replace the subscript $c$ by the subscript $S$ to indicate the full subcategory spanned by the complexes with $S$-constructible cohomology sheaves.

 
 Let $D_c^b(X;R)$ denote the cohomologically-bounded constructible derived category of sheaves on $X$ with coefficients in $R$-modules for some commutative ring $R$. Here, as usual, constructible means that the cohomology sheaves are constructible sheaves with respect to some algebraic Whittney stratification of $X$ and that they have finitely generated stalks. For such a sheaf 
 \[
 \F=(\cdots \to \F_{-1}\to \F_0\to \F_1\to \cdots),
 \]
 we let $H^n(\F)$ denote the $n^{th}$ cohomology sheaf. For a morphism $f:X\to X'$ of varieties, we have the usual functors
 \[
 f_\ast,\: f_!,\:, f^\ast,\:f^! 
 \]
 along with $\otimes^!$ and the internal Hom-functor $\underline{\Hom}(-,-).$ We also denote by $\mathbb{D}$ the Verdier duality functor on $D_c^b(X;R)$.
 
 For any $R$-module $M$, we denote by $\underline{M}=\pi^\ast M$ the constant sheaf on $X$ associated to $M$ placed in homological degree zero, where $\pi:X\lra \{pt\}$ is the unique map to a point.
 
\subsection{Equivariant derived category} Now assume that $X$ is equipped with an algebraic action of $\vp$ and let $D^b_{\vp,c}(X;\bO)$ denote the bounded constructible $\vp$-equivariant derived category of sheaves (for the analytic topology) of $\bO$-modules on $X$ in the sense of \cite{BernsteinLunts}.
 
Let $X^\vp$ denote the fixed-point subvariety. For much of Sections \ref{Section: sheaves} and \ref{Section: Tate-parity}, we need to assume that the $\vp$-action is trivial. We will denote the variety by $Y$ in this case, the main example being $Y=X^\vp$.  In this case, we have an equivalence 
\begin{align}\label{eqn: equivar coeff}
D^b_{\vp,c}(Y;\bO) \cong D^b_c(Y;\bO[\vp]).
\end{align}

\section{The Tate category}\label{sec:const}
In this section, we construct the integral Tate category by following the construction of \cite{Tr}, replacing $\bF$ by $\bO$. We then study \emph{parity objects} in this category, which is essentially the theory of Tate-parity sheaves over a point. To begin, we recall the notion of weakly-injective modules.
\begin{Def}
A $\bO[\vp]$-module $M$ is \emph{weakly injective} if there is a $\bO$-module $V$ such that $M\cong\bO[\vp]\otimes_{\bO} V$.
\end{Def}
Stated differently, $M$ is weakly-injective if it is injective relative to the trivial subgroup of $\vp$; see \cite[Section 2]{BIK} for a more general discussion. We note that weakly-injective modules are acyclic with respect to the functor of $\vp$-invariants. These modules are the appropriate analogues of free $\ff[\vp]$-modules arising in \cite{Tr}.


Now let $D^b(\bO[\varpi])$ be the bounded derived category of $\bO[\varpi]$-modules, and let $D^b_c(\bO[\vp])$ denote its full subcategory consisting of complexes with finitely generated cohomology modules. Set $\mathrm{Perf}_c(\bO[\varpi])$ to be the thick subcategory generated by finitely generated {weakly injective} $\bO[\varpi]$-modules. Note that $\mathrm{Perf}_c(\bO[\varpi])$ forms a tensor ideal.

\begin{Def} The \emph{integral Tate category} is the Verdier quotient 
\begin{align*}
\T_\infty\modc:= D^b_c(\bO[\varpi])/\mathrm{Perf}_c(\bO[\varpi]).
\end{align*}
\end{Def}
It is naturally a triangulated category, the distinguished triangles being those which are isomorphic to the image of a distinguished triangle in $D^b_c(\bO[\vp])$. We have the triangulated projection functor
\begin{align*}
\bT^*:D^b_c(\bO[\vp])\lra\T_\infty\modc.
\end{align*}

\subsection{Periodicity} Below we will introduce a natural $\zz/2\zz$-graded cohomology functor on the category $\T_\infty\modc$ referred to as \emph{Tate cohomology}. This gives an equivalence between $\T_\infty\modc$ and the category of $\zz/2\zz$-graded $\ff$-vector spaces (see Corollary \ref{Prop: Grothendieck} below). 

We begin by showing that $\T_\infty\modc$ is $2$-periodic.

\begin{Prop}\label{Prop: two periodic}
There is a natural isomorphism $[2]\cong[0]$ of functors on $\T_\infty\modc$.
\end{Prop}
\begin{proof}
The exact sequence 
\begin{align*}
0\to\bO\xrightarrow{N}\bO[\vp]\xrightarrow{1-g}\bO[\vp]\xrightarrow{\epsilon}\bO\to 0,
\end{align*}
where $N: \bO\lra \bO[\vp]$ denotes the multiplication by the norm $N$, gives a morphism
\begin{align}\label{eqn: deg 2 shift}
\bO\to\bO[2]
\end{align}
in $D^b_c(\bO[\vp])$, the cone of which lies in $\mathrm{Perf}_c(\bO[\varpi])$. Thus, this morphism becomes an isomorphism after applying $\bT^*$. The result follows by identifying the shift functor $[2]$ with a functor of tensoring over $\bO$ with $\bO[2]$.\end{proof}

For a 2-periodic triangulated category such as this, it makes sense to consider the $n$-fold homological shift for $n\in \zz/2\zz$. We denote the functor as $[n]$.

\subsection{Derived invariants}Recall the counit map
\begin{align*}
    \ep:\zvp&\lra \bO,\\
        g&\longmapsto 1,
\end{align*} and consider the derived functor of invariants
\begin{align*}
 \epsilon_\ast:=\underline{\Hom}(\bO,-):D^b_c(\zvp)\to D^+(\bO).
\end{align*}
More concretely, this functor is the composition
\begin{align*}
D^b_c(\zvp)\xrightarrow{I}K^+_c(Inj(\zvp))\xrightarrow{ \epsilon_\ast}K^+_c(Inj(\bO))\xrightarrow{Q}D^+_c(\bO),
\end{align*}
where $Q$ is the localization functor and $I$ comes by restriction from a right adjoint to the localization functor $K^+_c(Inj(\zvp))\to D^+_c(\zvp)$. Here $Inj(-)$ denotes the category of injective objects, $K^+$ denotes the bounded below homotopy category, $D^+$ denotes the (cohomologically) bounded below derived category, and the subscript $c$ indicates passing to the full subcategory of complexes with finitely generated cohomology modules.  On the level of complexes, $I$ sends an object in the derived category to an injective resolution.

\begin{Lem}
The derived invariants functor $ \epsilon_\ast$ preserves the finite generation of cohomology.
\end{Lem}
\begin{proof}
Consider the weakly injective resolution
\begin{align*}
\mathfrak{i}=(0\to\zvp\xrightarrow{1-g}\zvp\xrightarrow{N}\zvp\xrightarrow{1-g}\ldots)
\end{align*}
of $\bO$, where $N=\sum_{h\in\vp}h$ is the norm element.  For a bounded below complex 
\begin{align*}
B=(0\to B_0\to B_1\to\ldots)
\end{align*}
in $K^+(Inj(\zvp))$, the natural morphism from $B$ to the totalization of the double complex $B\otimes_{\bO}\mathfrak{i}$ is an isomorphism. Since totalization commutes with invariants, we see that $ \epsilon_\ast B$ is isomorphic to the totalization of the double complex $B\otimes_{\zvp}\mathfrak{i}$ written out below:
\[
    \begin{tikzcd}
    &\vdots&\vdots&\\
    0\ar[r]&B_0\ar[r]\ar[u,"1-g"]&B_1\ar[r]\ar[u,"1-g"]&\cdots\\
    0\ar[r]&B_0\ar[u,"N"]\ar[r]&B_1\ar[u,"N"]\ar[r]&\cdots\\
    0\ar[r]&B_0\ar[r]\ar[u,"1-g"]&B_1\ar[r]\ar[u,"1-g"]&\cdots\\
    &0\ar[u]&0\ar[u]&
    \end{tikzcd}
\]
Here the notation $1-g$ and $N$ in the complex indicates the action of elements of $\zvp$ on the $\zvp$-modules $B_i$. Using the horizontal-vertical spectral sequence, along with the fact that $\zvp$ is Noetherian, it follows that the associated graded module of $H^n \epsilon_\ast B$ with respect to the induced filtration is a sub-quotient of a finitely generated $\zvp$-modules, hence is finitely generated. Thus, the cohomology is itself finitely generated.
\end{proof}

\subsection{Tate cohomology} Now suppose that $B$ is an arbitrary bounded-below complex of $\bO[\vp]$-modules with finitely generated cohomology modules. The totalization of the double complex $B\otimes_{\bO}\mathfrak{i}$ is quasi-isomorphic to $B$ and is weakly injective. It follows that its invariant sub-complex, as written above, computes the cohomology of $ \epsilon_\ast B$. When $B$ is bounded, the periodicity of the double complex immediately implies the following lemma.
\begin{Lem}
If $B$ is bounded then these cohomology groups become $2$-periodic for large $n$; that is, the natural map 
\[
H^n \epsilon_\ast B\lra H^{n+2} \epsilon_\ast B
\]
induced by the morphism $B\lra B[2]$ induced from (\ref{eqn: deg 2 shift}) is an isomorphism for large $n$. 
\end{Lem}
\begin{Def}
 We define the \emph{Tate cohomology} functors $T^0$ and $T^1$ to be the resulting periodic cohomology groups. That is,
\begin{align*}
\left.\begin{matrix}
T^0&:=&\colim_{n}H^{2n} \epsilon_\ast\\
T^1&:=&\colim_nH^{2n+1} \epsilon_\ast\end{matrix}\right\}:D^b_c(\zvp)\to \bO\modc.
\end{align*}
\end{Def}

The preceding discussion shows that the colimits converge in finite time. Furthermore, the vertical-horizontal spectral sequence shows that cohomology groups $H^n \epsilon_\ast B$ vanish on perfect complexes for large $n$. It follows that the Tate cohomology functors factor through $\T_\infty\modc$; we abuse notation and denote the resulting functors as $T^0$ and $T^1$. For a distinguished triangle 
\[M\to N\to O\xrightarrow{+1}
\]
in $\T_\infty\modc,$ we obtain a $6$-periodic long exact sequence
\begin{align*}
\ldots\to T^0M\to T^0N\to T^0O\to T^1M\to T^1N\to T^1O\to\ldots.
\end{align*}

\subsection{Tate complex}\label{Section: tate complex} For computational reasons, it is useful to have a genuine complex whose the cohomology groups compute Tate cohomology. To this end, consider the 2-periodic acyclic complex 
\begin{align}\label{eqn: tate complex}
\liet=(\cdots\xrightarrow{N}\zvp\xrightarrow{1-g}\zvp\xrightarrow{N}\zvp\xrightarrow{1-g}\cdots).
\end{align}  
If $B=(0\to B_0\to \ldots\to B_n\to0)$ is bounded, then it is clear that one may compute $\liee B,\lieo B$ as the cohomology of the invariant sub-complex of the totalization of $B\otimes_{\bO}\liet$. This gives the totalization of the double complex 
\[
    \begin{tikzcd}
    &\vdots&&\vdots&\\
    0\ar[r]&B_0\ar[r]\ar[u,"1-g"]&\cdots\ar[r]&B_n\ar[r]\ar[u,"1-g"]&0\\
    0\ar[r]&B_0\ar[u,"N"]\ar[r]&\cdots\ar[r]&B_n\ar[u,"N"]\ar[r]&0\\
    0\ar[r]&B_0\ar[r]\ar[u,"1-g"]&\cdots\ar[r]&B_n\ar[r]\ar[u,"1-g"]&0\\
    &\vdots\ar[u]&&\vdots\ar[u]&
    \end{tikzcd}
\]
where the notation of the vertical maps indicates the action of elements of $\zvp$ on the $\zvp$-modules $B_i$ as above. Clearly this totalization coincides with the totalization of $B\otimes_{\zvp}\mathfrak{i}$ in large positive degrees. This allows us to work with explicit complexes rather than computing colimits. Both constructions easily generalize to the sheaf-theoretic versions studied in the next section.%
\begin{Lem}\label{Lem: valued in Fp mod}
The Tate cohomology functors $T^0,T^1$ are valued in the category of $\ff$-modules.
\end{Lem}
\begin{proof}
It is enough to check that $T^0\bO$ and $T^1\bO$ are $\ff$-modules. Utilizing the Tate complex $\ft$, we need to compute the cohomology of the complex
\[
\cdots\xrightarrow{N}\bO\xrightarrow{1-g}\bO\xrightarrow{N}\bO\xrightarrow{1-g}\cdots,
\]where $g$ acts trivially on $\bO$. It follows that
\[
T^0\bO=\ff,\:\text{ and }\:T^1\bO=0.\qedhere
\]
\end{proof}



\subsection{Parity objects} We now introduce the notion of Tate-parity objects. This is what ultimately plays the role of the JMW Condition \ref{parcon} in the construction of Tate-parity sheaves.

\begin{Def}
We say that an object $M$ of $\T_\infty\modc$ is \emph{Tate-even} (resp. \emph{Tate-odd}) if $T^1M=0$ (resp. $T^0M=0$).  We say that $M$ is \emph{Tate-parity} if it is a direct sum of an odd and an even object. 
\end{Def}

For example, the image of $\bO$ in $\T_\infty\modc$ is Tate-even, while the exact sequence
\[
0\lra \bO\xrightarrow{\cdot N}\bO[\vp]\lra \bO[\vp]/N\lra 0
\]
implies that $\bT^\ast\bO[\vp]/N\cong \bT^\ast\bO[1]$ is Tate-odd. We have the following fundamental fact:
\begin{Lem}\label{allparity}
\begin{enumerate}\item Every object of $\T_\infty\modc$ is Tate-parity.\
\item If $M$ is Tate-even and $N$ is Tate-odd then $\Hom_{\T_\infty\modc}(M,N)=0$.\
\item If $M$ is Tate-even then $M\cong \bT^*\bO^k$ for some non-negative integer $k$.\end{enumerate}
\end{Lem}

\begin{proof}We will give an argument in the language of homotopy theory, but remark that one can give a more elementary proof using stable module categories. As we have already mentioned, $\T_\infty\modc$ is equivalent to the homotopy category of a certain category of module spectra over a certain $\mathbb{E}_\infty$-ring spectrum $\T_\infty$ (see \cite{Tr}). The homotopy groups of $\T_\infty$ are equal to the Tate cohomology groups of the trivial $\bO[\vp]$-module, so that Lemma \ref{Lem: valued in Fp mod} implies
\[
\pi_\bullet (\T_\infty) \cong \ff[t,t^{-1}],
\]
where $t$ is in degree $2$. Therefore, $\T_\infty$ is a \emph{good coefficient algebra} (see \cite{Tr1}). For an $\mathbb{E}_\infty$-algebra $E$ we say that a module spectrum $M\in \Mod({E})$ is \emph{even} if all its odd homotopy groups vanish. We recall the following result for module spectra of such algebras.

\begin{Prop}\cite[Prop. 2.1]{Tr1}
Let $E$ be a good coefficient algebra. Then every object $M\in \Mod({E})$ is isomorphic to a direct sum $M_0\oplus \Sigma M_1$, where $M_i$ are both even. Moreover, if $M$ and $N$ are even, then there are natural isomorphisms
\begin{align*}
\begin{matrix}
[M,N] &\cong&\Hom_{\pi_0{E}}(\pi_0M,\pi_0N),\\
[M,\Sigma N]& \cong&\Ext^1_{\pi_0{E}}(\pi_0M,\pi_0N).\end{matrix}
\end{align*}
\end{Prop}

Parts (1) and (2) of the lemma follow, since a summand of a compact object is compact and $\pi_0(\T_\infty)=\ff$ is a field of odd characteristic. For part (3), let $M$ be Tate-even and choose a basis of the finite-dimensional vector space $T^0M$, with cardinality $k$ say. Since $T^0M=\Hom_{\T_\infty\modc}(\bT^*\bO,M)$, the choice of basis induces a map $\bT^*\bO^k\to M$ which becomes an isomorphism after applying $T^0$. The cone is trivial under both $T^0$, $T^1$ and so must be $0$.\end{proof}

\begin{Rem}\label{Rem: why not two}
This argument is where the assumption $p\neq 2$ is used. Indeed, $\Ext_{\ff_2}^1(\ff_2,\ff_2)\neq0$ so that Part (2) fails in this case. Our restriction to $p>2$ is motivated by the need to have a Tate-theoretic version of the JMW condition (\ref{parcon}). This lemma provides this condition. 
\end{Rem}

The following is an immediate consequence:
\begin{Cor}\label{Prop: Grothendieck}\
The functor $(T^0,T^1)$ is an equivalence between $\T_\infty\modc$ and the category of $\zz/2\zz$-graded finite dimensional $\ff$-vector spaces.\
\end{Cor}

\begin{Rem}\label{Rem: modular coefficients}
Corollary \ref{Prop: Grothendieck} implies there must be an isomorphism 
\begin{align*}
\bT^*\ff\cong\bT^*\bO\oplus\bT^*\bO[1].
\end{align*}
It is an enjoyable exercise to find an explicit isomorphism. 
\end{Rem}

\section{The integral Tate category}\label{Section: sheaves}

For the entirety of this section, we assume that $Y$ is a complex algebraic variety which we equip with the trivial action of $\vp$. Our present goal is to develop the theory of sheaves on $Y$ generalizing the category $\T_\infty\modc$ from the previous section. We also study two forms of Tate-cohomology, one valued in the category of constructible sheaves over $\ff$ (see \ref{Section: tate sheaves}) and the other an analogue of hypercohomology (see \ref{hyperco}).

\subsection{The construction} Consider the equivariant derived category $D^b_{\vp,c}(Y;\bO)$ and recall from (\ref{eqn: equivar coeff}) that we have an equivalence 
\begin{align*}
D^b_{\vp,c}(Y;\bO) \cong D^b_c(Y;\bO[\vp]).
\end{align*}
 Let $\mathrm{Perf}_c(Y; \bO[\vp])$ denote the thick subcategory of $D^b_c(Y;\bO[\vp])$ generated by sheaves of constructible $\zvp$-modules whose stalks are all weakly injective. An object of $\mathrm{Perf}_c(Y; \bO[\vp])$ is called a \emph{perfect} complex. As over a point, $\mathrm{Perf}_c(Y; \bO[\vp])$ is a tensor ideal in $D^b_{\vp,c}(Y;\bO)$.

\begin{Def} The \emph{constructible integral Tate category} is the Verdier quotient 
\begin{align*}
\mathrm{Sh}_c(Y;\T_\infty):=D^b_c(Y;\bO[\vp])/\mathrm{Perf}_c(Y;\bO[\vp]).
\end{align*}
\end{Def}
We write $\bT^*:D^b_c(Y;\bO[\vp])\to \mathrm{Sh}_c(Y;\T_\infty)$ for the projection functor. If $S$ is a fixed stratification of $Y$, we replace the subscript $c$ by the subscript $S$ in any of the above categories to indicate the full thick subcategory generated by sheaves which are constructible along $S$ and, in the case of $\mathrm{Perf}_S(Y;\bO[\vp])$, have weakly injective stalks. 

\begin{Rem}\label{equivoo}It is not immediate whether the natural functor
\begin{align*}
D^b_S(Y;\bO[\vp])/\mathrm{Perf}_S(Y;\bO[\vp])\lra \mathrm{Sh}_S(Y;\T_\infty)
\end{align*}
is an equivalence. We will show this in Corollary \ref{equivo}.\end{Rem} 

\begin{Rem}\label{broker?}In \cite{Tr}, a slightly different definition of perfect complexes is given: by strict analogy, we ought to say that a complex is perfect if its stalks are all isomorphic to $0$ in $\T_\infty\modc$. Certainly our notion of perfect complex is contained within this. As we will see in Corollary \ref{fixer!}, these two definitions do agree.\end{Rem}

\subsection{Tate cohomology sheaves}\label{Section: tate sheaves} The analogue of Proposition \ref{Prop: two periodic} holds for $\mathrm{Sh}_c(Y;\T_\infty)$, with essentially the same proof. In particular, we may also define the functors $T^0,T^1$. 
\begin{Def}\label{Def: tate cohom sheaves}For a complex $\F$ in $D^b_c(Y;\zvp)$ we define the \emph{Tate cohomology sheaves} of $\F$ as the colimits of the (eventually constant) systems:
\begin{align*}
\left.\begin{matrix}
T^0&:=&\colim_{n}H^{2n} \epsilon_\ast\\
T^1&:=&\colim_nH^{2n+1} \epsilon_\ast\end{matrix}\right\}:D^b_c(Y;\zvp)\to \mathrm{Sh}_c(Y;\ff)
\end{align*}
\end{Def}
As in Section \ref{sec:const}, we may also compute these cohomology functors for any bounded complex $\cB$ quasi-isomorphic to $\F$ and take the cohomology of the 2-periodic complex $\cB\otimes_{\zvp}\liet$, where $\liet$ is the Tate complex.

A priori $T^0,T^1$ are valued in constructible sheaves of $\bO$-modules, but they are easily seen to be compatible with stalks and hence Lemma \ref{Lem: valued in Fp mod} implies they take values in constructible sheaves of $\ff$-modules. Likewise, for a fixed stratification $S$ we have the functors
\begin{align*}
T^0,T^1:\mathrm{Sh}_S(Y;\T_\infty)\to \mathrm{Sh}_S(Y;\ff).
\end{align*}
Finally, for a distinguished triangle $\mathcal{F}\to\mathcal{G}\to\mathcal{H}\xrightarrow{+1}$ in $\mathrm{Sh}_c(Y;\T_\infty)$, we have a $6$-periodic long exact sequence
\begin{align*}
\ldots\to T^0\mathcal{F}\to T^0\mathcal{G}\to T^0\mathcal{H}\to T^1\mathcal{F}\to T^1\mathcal{G}\to T^1\mathcal{H}\to\ldots.
\end{align*}

\subsection{The six functors}\label{Rem: six functors}We remark that the six functors for $D^b_{\vp,c}(-;\bO)$ all send perfect complexes to perfect complexes (when $\vp$ acts trivially on the underlying spaces). In particular, we obtain via descending along the localization $\bT^\ast$ a notion of the six functors on $\mathrm{Sh}_c(-;\T_\infty)$; we shall use the standard notations:
\begin{align*}
f_*,f^*,f_!,f^!,\mathbb{D},\underline{\Hom}
\end{align*}
for these. We also obtain a tensor product
\[
\F_1\otimes \F_2=\mathbb{D}\underline{\Hom}(\F_1,\mathbb{D}\F_2).
\]The proof is essentially the same as for $\mathrm{Sh}_c(-,\T_0)$ given in \cite{Tr}. Although it is not stated in \emph{loc. cit.}, it is a formal consequence that the usual adjunctions hold. 

Note that we have briefly switched notation from $D^b_{c}(-;\zvp)$ to $D^b_{\vp,c}(-;\bO)$ in order to emphasize that the Verdier duality and tensor product functors are taken $\bO$-linearly (and given, respectively, the inverse and diagonal $\vp$-equivariant structure). If we were to dualize or tensor instead over $\zvp$, we would not preserve boundedness.

\subsection{Extending scalars}\label{Sec: scalars} We may compute the Tate cohomology sheaves in the following case: recall the \emph{Tate extension of scalars functor}
\begin{equation}\label{eqn: tate base change}
    \bT: D^b_c(Y;\bO)\lra \mathrm{Sh}_c(Y;\T_\infty)
\end{equation}
given by the composition $\bT:=\bT^\ast \epsilon^\ast$ as in the introduction\footnote{This functor is analogous to the functor $-\otimes_{\ff}\T_0$ given in \cite{Tr}.} 

\begin{Lem}\label{Lem: cohom calc}
Let $\F$ be an object of $D^b_c(Y;\bO)$ whose cohomology sheaves $H^n\F$ have $\bO$-free stalks. Then we have a natural isomorphism
\begin{align*}
T^i(\bT\F)\cong\bigoplus_{n\in i}\ff\otimes_{\bO}H^n\F
\end{align*}
for any $i\in\zz/2\zz$. 
\end{Lem}

\begin{proof}Fix a bounded complex $\cB$ quasi-isomorphic to $\F$, which we further assume has torsion-free stalks. This is possible by our assumption that $H^i\F$ are $\bO$-free. We have a canonical isomorphism 
\begin{align*}
 \epsilon^*\cB\otimes_{\zvp}\liet \cong \bigoplus_{n\in\zz}\cB\otimes_{\bO}\ff[2n]
\end{align*}
and therefore we identify
\begin{align*}
T^i\bT\F\cong\bigoplus_{n\in i}H^{n}(\cB\otimes_{\bO}\ff).
\end{align*}
There is a Cartan--Eilenberg spectral sequence converging to $H^{\bullet}(\cB\otimes_{\bO}\ff)$ whose $E_2$ page equals $H^\bullet(\cB)\otimes^\bullet  \ff$. Since $\bO$ has homological dimension $1$, the spectral sequence degenerates here for any $\cB$. If further each $H^n(\cB)$ has $\bO$-free stalks, then the $E_2$-page lives in a single row, so that the associated filtration is trivial and we obtain the desired equality.
\end{proof}

\subsection{Hom-sets as colimits}There is a relatively simple description of hom-sets in the Tate category. 

\begin{Prop}\label{colimy}For any complexes $\F$ and $\G$ in $D^b_c(Y;\zvp)$, there is a natural isomorphism
\[
\Hom_{\mathrm{Sh}_c(Y;\T_\infty)}(\bT^*\F,\bT^*\G)\cong\colim_n\Hom_{D^b_c(Y;\zvp)}(\F,\G[2n]).
\]\end{Prop}
\begin{proof}
Recall that in $D^b_c(Y;\zvp)$ we have a natural transformation $[0]\to[2]$, which becomes an isomorphism after applying $\bT^*$. It follows by functoriality that for any $\F,\G$ in $D^b_c(Y;\zvp)$, the localization $\bT^*$ induces a map
\begin{align*}
C:\colim_n\Hom_{D^b_c(Y;\zvp)}(\F,\G[2n])\lra \Hom_{\mathrm{Sh}_c(Y;\T_\infty)}(\bT^*\F,\bT^*\G),
\end{align*}
which is functorial in $\F$ and $\G$. Moreover, morphisms on the left can be composed in the natural manner, and $C$ respects composition. We claim that $C$ is an isomorphism.

Consider the category $\widetilde{\mathrm{Sh}}_c$ whose objects are the same as those of $D^b_c(Y;\zvp)$ with morphisms given by
\begin{align*}
\Hom_{\widetilde{\mathrm{Sh}}_c}(\F,\G):=\colim_n\Hom_{D^b_c(Y;\zvp)}(\F,\G[2n]).
\end{align*}
Then $\widetilde{\mathrm{Sh}}_c$ is a triangulated category, and $\bT^*$ factors as the composition of the triangulated functors
\begin{align*}
D^b_c(Y;\zvp)\xrightarrow{\widetilde{\bT^*}} \widetilde{\mathrm{Sh}}_c\xrightarrow{C} \mathrm{Sh}_c(Y;\T_\infty).
\end{align*}

We claim that $\widetilde{\bT^*}$ kills perfect complexes. To see this, it is enough to show that for any perfect complex $\cP$, we have 
\[
\Hom_{D^b_c(Y;\zvp)}(\cP,\cP[n])=0
\] for all $n$ large enough. In fact, we will show that for any $\cF$ in $D^b_c(Y;\zvp)$, we have $\Hom_{D^b_c(Y;\zvp)}(\cF,\cP[n])=0$ for all $n$ large enough. 

By Verdier duality and since $\mathrm{Perf}_S(Y;\zvp)$ is a tensor ideal, it suffices to consider the case when $\F$ is the constant sheaf $\underline{\bO}=\pi^\ast\bO$, where $\pi:Y\to \{pt\}$ denotes the constant map. Taking stalks, the claim then follows from the fact that a perfect complex in $D^b_c(\zvp)$ has bounded derived invariants. This holds as any such object is a finite iterated cone of weakly injective (in particular, $ \epsilon_\ast$-acyclic) $\zvp$-modules.

It follows that the functor $\widetilde{\bT^*}$ factors as a composition:
\begin{align*}
D^b_c(Y;\zvp)\xrightarrow{{\bT^*}}\mathrm{Sh}_c(Y;\T_\infty)\xrightarrow{J} \widetilde{\mathrm{Sh}}_c 
\end{align*}
such that $C\circ J\cong id$. But $J$ is full, since
\[
\Hom_{\mathrm{Sh}_c(Y;\T_\infty)}(\bT^*\F,\bT^*\mathcal{G}[2n])=\Hom_{\mathrm{Sh}_c(Y;\T_\infty)}(\bT^*\F,\bT^*\mathcal{G})
\]
for all $n$. 
\end{proof}

We are now able to demonstrate Remark \ref{equivoo}. Fix a stratification $S$ of $Y$ and recall that the notations 
\[
D^b_S(Y;\zvp),\quad \mathrm{Sh}_S(Y;\T_\infty),\quad\mathrm{Perf}_S(Y;\zvp)
\] denote, respectively, the full thick subcategories of 
\[
D^b_c(Y;\zvp), \quad \mathrm{Sh}_c(Y;\T_\infty), \quad\mathrm{Perf}_c(Y;\zvp)
\]generated by all sheaves constructible along $S$ with, in the case of $\mathrm{Perf}_S(Y;\zvp)$, weakly injective stalks.

\begin{Cor}\label{equivo}The functor 
\begin{align*}
D^b_S(Y;\zvp)/\mathrm{Perf}_S(Y;\zvp)\to \mathrm{Sh}_S(Y;\T_\infty)
\end{align*}
is an equivalence.\end{Cor}

\begin{proof}It is certainly essentially surjective. It still holds that $\mathrm{Perf}_S(Y;\zvp)$ is a tensor ideal of $D^b_S(Y;\zvp)$ containing the complex 
\[
0\to\underline{\zvp}\xrightarrow{N}\underline{\zvp}\to0.
\]This implies that 
\[
D^b_S(Y;\zvp)/\mathrm{Perf}_S(Y;\zvp)
\]
is still $2$-periodic and the above calculation of hom-sets works out in exactly the same way for the left-hand side of the claimed equivalence. This shows the functor in question is fully faithful.
\end{proof}

\subsection{Image of Tate extension of scalars}

For sheaves in the image of the Tate extension of scalars functor $\bT$, the computation of hom-sets simplifies dramatically. Write $\bF$ for the modular reduction functor; it is left adjoint to the inclusion functor, and right adjoint to the $[-1]$-shift of the inclusion. Temporarily writing $I$ for the inclusion functor, we have:
\begin{align*}
I\bF(-)=\pi^*\ff\otimes(-)=\underline{\Hom}(\pi^*\ff,-)[1].
\end{align*}

\begin{Prop}\label{moddyprop}Let $\F$, $\G$ be objects of $D^b_c(Y;\bO)$. Then we have
\begin{align*}
\Hom_{\mathrm{Sh}_c(Y;\T_\infty)}(\bT\F,\bT\G)=\bigoplus_{i\in\zz}\Hom_{D^b_c(Y;\ff)}(\bF\F,\bF\G[2i]).
\end{align*}
\end{Prop}

\begin{proof}Freely using the fact that the six functors commute with $ \epsilon^*$ and $\bT^*$ and hence with $\bT$, we have 
\begin{align*}
\begin{matrix*}[l]\Hom_{\mathrm{Sh}_c(Y;\T_\infty)}(\bT\F,\bT\G)& =& \Hom_{\mathrm{Sh}_c(Y;\T_\infty)}(\underline{\bO},\bT\underline{\Hom}(\F,\G))\\
&=&\Hom_{\T_\infty\modc}(\bO,\bT\pi_*\underline{\Hom}(\F,\G))\\
&=&T^0\bT\pi_*\underline{\Hom}(\F,\G)\\
&=&H^0( \epsilon^*\pi_*\underline{\Hom}(\F,\G)\otimes_{\zvp}\liet),\\
&=&H^0(\bigoplus_{i\in\zz}(\pi_*\underline{\Hom}(\F,\G)\otimes_{\bO}\ff[2i]))\\
&=&\bigoplus_{i\in\zz}H^{2i}(\pi_*\underline{\Hom}(\F,\G)\otimes_{\bO}\ff).
\end{matrix*}
\end{align*}

Then, by the projection formula we have 
\begin{align*}
\begin{matrix*}[l]H^{2i}(\pi_*\underline{\Hom}(\F,\G)\otimes_{\bO}\ff)&=&H^{2i}(\pi_*(\underline{\Hom}(\F,\G)\otimes\pi^*\ff))\\
&=&H^{2i}(\pi_*\underline{\Hom}(\pi^*\ff,\underline{\Hom}(\F,\G))[1])\\
&=&H^{2i}(\pi_*\underline{\Hom}(\pi^*\ff\otimes\F,\G)[1])\\
&=&\Hom_{D^b_c(Y;\bO)}(\pi^*\ff\otimes\F,\G[2i+1])\\
&=&\Hom_{D^b_c(Y;\bO)}(I\bF\F,\G[2i+1])\\
&=&\Hom_{D^b_c(Y;\ff)}(\bF\F,\bF\G[2i]),
\end{matrix*}
\end{align*}
which completes the calculation.\end{proof}

\subsection{Relation to modular reduction}\label{addendum}By functoriality, $\bT$ induces maps
\begin{align*}
\bT:\Hom_{D^b_c(Y;\bO)}(\F,\G)\lra \Hom_{\mathrm{Sh}_c(Y;\T_\infty)}(\bT\F,\bT\G)
\end{align*}
which are compatible with all compositions. Tracing through the calculation of Proposition \ref{moddyprop}, we have the commutative diagram
\[
\begin{tikzcd}
\Hom_{D^b_c(Y;\bO)}(\F,\G)\ar[rr,"\bT"]\ar[rd,swap,"\ff"]&&\Hom_{\mathrm{Sh}_c(Y;\T_\infty)}(\bT\F,\bT\G)\ar[ld,"p"]\\
&\Hom_{D^b_c(Y;\ff)}(\bF\F,\bF\G),&
\end{tikzcd}
\]
where the arrow $p$ is the projection of the $i=0$ summand
\begin{align*}
\Hom_{\mathrm{Sh}_c(Y;\T_\infty)}(\bT\F,\bT\G)=\bigoplus_{i\in\zz}\Hom_{D^b_c(Y;\ff)}(\bF\F,\bF\G[2i])\xrightarrow{p}\Hom_{D^b_c(Y;\ff)}(\bF\F,\bF\G).
\end{align*}

%

\subsection{The Tate hypercohomology spectral sequence}\label{hyperco} If we equip the point $\{pt\}$ with a trivial $\vp$-action and denote by $\pi:Y\to \{pt\}$ the constant map, we have an induced functor (see the discussion after \cite[Proposition 4.3]{Tr})
\[
\pi_\ast: \mathrm{Sh}_c(Y;\T_\infty)\lra\mathrm{Sh}_c(\{pt\};\T_\infty)=\T_\infty\modc;
\]
composing with the Tate cohomology functors $T^\bullet$ on $\T_\infty\modc$ gives a notion of \emph{Tate hypercohomology} $\mathrm{TH}^\bullet$. The goal of this section is to give a spectral sequence 
\[
E^{p,q}_2(\F)=H^q(Y,T^p(\F))\Rightarrow TH^{p+q}\F
\]
abutting to this cohomology such that $E^{\bullet,\bullet}_n(\F)$ is a bi-graded module of $E^{\bullet,\bullet}_n(\underline{\bO})$ for each $n$. Making this precise requires introducing a certain auxiliary category  mapping to $\T_\infty\modc$, which we discuss below. We then construct a spectral sequence (see Lemma \ref{Lem: spectral}) relating it to the Tate cohomology sheaves. This will be used in subsequent sections to study indecomposible objects in $\mathrm{Sh}_c(Y;\T_\infty)$; see Theorem \ref{megathm}.

Consider the dg-algebra over $\zvp$
\begin{align*}
E:=(\pi_*\underline{\bO})^{op}.
\end{align*}
More precisely, we set $E=\pi_*^0\cC$ for any coconnective $\pi_*$-acyclic dg-algebra $\cC$ in $D^b_c(Y;\zvp)$ equipped with a map $\underline{\bO}\to \cC$ of dg-algebras that is also a quasi-isomorphism. Here $\pi_*^0$ denotes the underived global sections. Such a $\cC$ certainly exists, and may be chosen to be bounded since $\pi_*$ has finite cohomological dimension. 

Let $E\moddg$ denote the full subcategory of the derived category of dg-modules over $E$ consisting of those dg-modules whose underlying complex is in $D^b_c(\zvp)$. We note that $\cC$ may be taken to have trivial $\vp$-action, and therefore so may $E$. We emphasize, however, that $E$ is a dg-algebra over $\zvp$ by definition. In particular, $E\moddg$ consists of complexes of $\zvp$-modules on which the $\vp$-action is not necessarily trivial. 

The pushforward $\pi_*$ factors
\[
\begin{tikzcd}
D^b_c(Y;\zvp)\ar[rd,"\Pi_*"]\ar[rr,"\pi_\ast"]&& D^b_c(\zvp)\\
&E\moddg,\ar[ru,"For"]&
\end{tikzcd}
\]where $For: E\moddg\lra D^b_c(\zvp)$ is the forgetful functor.  For $\F$ in $D^b_c(Y;\zvp),$ the underlying object of $\Pi_\ast\F$ in $D^b_c(\zvp)$ is just $\pi_*\F$, but we give it a different name to indicate that we are remembering the structure of an $E$-dg-module. 

We say an $E$-dg-module is \emph{perfect} if it is weakly injective as an object of $D^b_c(\zvp)$; we let $\Perf_{dg}(E)$ the thick subcategory of $E\moddg$ generated by such objects. The functor $\Pi_*$ is triangulated and sends perfect complexes to perfect complexes since $\pi_*$ does. Thus, we obtain a functor
\begin{align*}
\Pi_*:\mathrm{Sh}_c(Y;\T_\infty)\lra E\moddg/\Perf_{dg}(E).
\end{align*} 
The induced restriction functor
 \[
 For: E\moddg/\Perf_{dg}(E)\lra \T_\infty\modc
 \]
 is right adjoint to $-\otimes E$, where the tensor product is the one discussed in Section \ref{Rem: six functors} (see also \cite[Section 4.1.1]{Tr}).
 
 Recalling that 
 \[
 T^nS= \Hom_{\T_\infty\modc}(\bT^\ast\bO,S[n])
 \]
 for $S$ in $\T_\infty\modc$, we see that the Tate cohomology of an object in $E\moddg/\Perf_{dg}(E)$ is naturally equipped with the structure of a graded $T^\bullet E$-module. In particular, for a complex $\F$ in $\mathrm{Sh}_c(Y;\T_\infty)$, the module $T^\bullet\Pi_*\F$ is a graded $T^\bullet E$-module. 
\begin{Def}
We define the \textbf{Tate hypercohomology} functors
\begin{align*}
\mathrm{TH}^\bullet:=T^\bullet\circ\Pi_*:\mathrm{Sh}_c(Y;\T_\infty)\lra T^\bullet E\text{-}\mathrm{Mod}_{gr} \:\text{ with }\bullet\in \{0,1\},
\end{align*}
 to be the composition of $\Pi_\ast$ with the Tate cohomology functors $T^0$ and $T^1$. 
\end{Def}
\begin{Rem}
In terms of objects, this functor does nothing more than take derived global sections, project to $\T_\infty\modc$, and take Tate cohomology. The dg-algebra $E$ and dg-module structures are technical tools to ensure that the resulting cohomology is a graded module for the "Tate cohomology" of $Y$, $\mathrm{TH}^\bullet(\underline{\bO})\cong T^\bullet E$ (see the discussion after the proof of Lemma \ref{Lem: spectral}). 
\end{Rem}

We now construct a spectral sequence converging to $\mathrm{TH}^\bullet$, the $E_2$-page of which is $R^\bullet\pi_*(T^\bullet(-))$. The construction is likely standard, but we have not found an explicit reference.

Fix a complex $\cF$ in $\mathrm{Sh}_c(Y;\T_\infty)$, and let $\F\to\mathcal{B}=(0\to\cB_0\to\cdots\to \cB_n\to0)$ be a bounded replacement of $\F$ such that each $\cB_i$ is $\pi_*$-acyclic and $\cB$ is a $\cC$-dg-module. Then $\Pi_*\F$ is isomorphic in $E\moddg$ to 
\[
B:=\pi_*\cB=(0\lra \pi_*^0\cB_0\lra\cdots\lra\pi_*^0\cB_n\lra0).
\]
Therefore, $\mathrm{TH}^\bullet \F$ is computed as the cohomology of the totalization of $B\otimes_{\zvp}\liet$ (see Section \ref{Section: tate complex}). This complex is the same as the global sections $\pi^0_*(\cB\otimes_{\zvp}\liet)$, since $\cB$ is bounded.

To simplify notation, let us write $\cA$ for the totalization of the complex $\cB\otimes_{\zvp}\liet$. Let $\cI$ be a Cartan--Eilenberg resolution of $\cA$ (for its existence, see \cite[tag: 015G]{stacks-project}). Thus we have a double complex
\[
    \begin{tikzcd}
    &\vdots&\vdots&\vdots&\\
    \cdots\ar[r]&\cI_{n-1,1}\ar[r]\ar[u]&\cI_{n,1}\ar[r]\ar[u]&\cI_{n+1,1}\ar[r]\ar[u]&\cdots\\
    \cdots\ar[r]&\cI_{n-1,0}\ar[u]\ar[r]&\cI_{n,0}\ar[r]\ar[u]&\cI_{n+1,0}\ar[u]\ar[r]&\cdots\\
    \cdots\ar[r]&\cA_{n-1} \ar[r]\ar[u]&\cA_{n}\ar[u] \ar[r]&\cA_{n+1} \ar[r]\ar[u]&\cdots\\
    &0\ar[u]&0\ar[u]&0\ar[u]&
    \end{tikzcd}
\]
such that the columns are exact and consist of injective objects above the row containing $\cA$. By construction, the columns of the $E_1$-page of the horizontal-vertical spectral sequence are resolutions of the Tate cohomology sheaves $T^n\F\cong T^n\cB$. It follows that the $E_2$-page of the horizontal-vertical spectral sequence of the double complex $\pi_*^0\cI$ is
\[
E_2^{p,q}(\pi_*^0\cI)=R^p\pi_*(T^q(\cB)).
\]

\begin{Lem}\label{Lem: spectral}
The spectral sequence 
\[
E_2^{p,q}(\pi_*^0\cI)=R^p\pi_*(T^q(\cB))
\]converges to $\mathrm{TH}^{p+q}\F$.
\end{Lem} 

\begin{proof}
We note first of all that the $E_2$-page is bounded vertically, so the spectral sequence converges as it is eventually constant. Moreover, this boundedness ensures that for any $p,q$, we have 
\[
E_\infty^{p,q}(\pi_*^0\cI)=E_\infty^{p,q}( \tau^{\geq  l}\pi_*^0\cI)
\]
for some $ l< p$. Here $ \tau^{\geq  l}\pi_*^0\cI$ denotes the double complex obtained from $\pi_*^0\cI$ by deleting all columns to the left of the $ l^{th}$ column. This truncated complex is the image under $\pi_*^0$ of a Cartan-Eilenberg resolution of the bounded-below complex $ \tau^{\geq  l}\cA$, so its horizontal-vertical spectral sequence converges to $R^{p+q}\pi_*\left({} \tau^{\geq  l}\cA\right)$. But $ \tau^{\geq  l}\cA$ is a bounded-below complex of $\pi_*$-acyclic objects, so we have 
\[
R^{p+q}\pi_*\left({} \tau^{\geq  l}\cA\right)=H^{p+q}\pi_*^0\left({{} \tau}^{\geq  l}\cA\right)=H^{p+q} \left( \tau^{\geq  l}(B\otimes_{\zvp}\liet)\right).
\]This coincides with 
\begin{align*}
H^{p+q}(B\otimes_{\zvp}\liet)=T^{p+q}\Pi_*\F
\end{align*}
in all degrees greater than $ l$, as required. 
\end{proof}

We will denote this spectral sequence $E^{\bullet,\bullet}_\bullet(\F):=\{E^{p,q}_n(\F)\}$. This is a slight abuse of notation since the pages $E_0, E_1$ depend on choice of Cartan-Eilenberg resolution $\cI^{\bullet,\bullet}$, though the $E_2$-page is independent (up to isomorphism) of this choice. As usual with spectral sequences, $E^{\bullet,\bullet}_n(\F)$ is a bi-graded dg-module for $E^{\bullet,\bullet}_n(\underline{\bO})$ for each $n\geq0$. 

An important example is when $H^\bullet(Y;\ff)$ is concentrated in even degrees. In this case, the module structure is easy to describe on the $E_2$-page. Indeed, the $E_2$-page of the spectral sequence associate to the constant sheaf $\underline{\bO}$ is concentrated in even bi-degrees and so has collapsed, giving the bi-graded algebra 
\[
E_2^{\bullet,\bullet}(\underline{\bO})=H^\bullet(Y;\ff)[t,t^{-1}]
\]
with $t$ in bi-degree $(2,0)$ and $H^\bullet(Y;\ff)$ equipped with its usual grading inserted vertically. That is,
\begin{align}\label{nice}
E^{p,q}_2(\underline{\bO})=
\begin{cases}H^q(Y;\ff)t^{p/2}&: p\text{ even,}\\ \qquad 0&: p\text{ odd}.\end{cases}
\end{align}
For any $\F$, $E^{p,q}_2(\F)$ is the horizontally 2-periodic double complex $H^q(Y;T^p \F)$ (with some differentials of bi-degree $(-1,2)$); the action of $H^q(Y;\ff)$ in the columns is just the usual action on cohomology with coefficients in the constructible sheaf $T^p \F$, and the action of $t$ is the obvious $2$-periodic structure.

\subsection{Tate-Support of a sheaf}

An important application of the Tate hypercohomology spectral sequence is the following characterization of when a sheaf is non-zero in $\mathrm{Sh}_c(Y;\T_\infty)$.
\begin{Prop}\label{suppprop}Let $\F$ be an object of $\mathrm{Sh}_c(Y;\T_\infty)$ whose stalks are all $0\in \T_\infty$-$mod_c$. Then $\F=0$.\end{Prop}
\begin{proof}Using the `stalk-preserving' triangles
\begin{align*}
j_!j^!\F\to \F\to i_*i^*\F\xrightarrow{+1}
\end{align*}
where $j:Y_\lam\hra Y$ is the inclusion of an open stratum and $i:\hra Y$ is the closed embedding of the compliment, we may reduce by induction to the case where $Y$ is smooth and $\F$ (or rather, an object of $D^b_c(Y;\zvp)$ underlying $\F$) is constructible with respect to the trivial stratification. It is enough then to show that $\End_{\mathrm{Sh}_c(Y;\T_\infty)}(\F)=0$. But we have:
\begin{align*}
\End_{\mathrm{Sh}_c(Y;\T_\infty)}(\F) = \Hom_{\mathrm{Sh}_c(Y;\T_\infty)}(\bT\underline{\bO},\underline{\Hom}(\F,\F))=\mathrm{TH}^0\underline{\Hom}(\F,\F).
\end{align*}

In general, the costalks of $\underline{\Hom}(\F,\G)$ are all $0$ if the stalks $\F$ or the costalks of $\G$ are. Since $Y$ is smooth, costalks and stalks coincide up to shift on constructible complexes. Therefore, our assumption implies that the stalks of $\underline{\Hom}(\F,\F)$ are all $0$ in the Tate category. In particular, the result follows from showing that $\mathrm{TH}^0\F=0$ whenever the stalks of $\F$ are all $0$. But taking stalks commutes with taking Tate cohomology sheaves, so that the Tate cohomology sheaves of $\F$ must also be $0$. In particular, the Tate hypercohomology spectral sequence vanishes term-wise on the $E_2$-page so that the result follows from Lemma \ref{Lem: spectral}.\end{proof}

For an object $\F$ of ${\mathrm{Sh}_c(Y;\T_\infty)}$, we follow the formalism of \cite{JMW1} and define the \emph{Tate support} of $\cF$, written $\supp_\T(\F)$, to be \emph{closure} of the set of points $i_y:\{y\}\to Y$ such that $i_y^*\F$ is non-zero as an object of $\T_\infty\modc$.

\begin{Rem}
This is in general smaller than the support of a complex representing $Y$. For example, the constant sheaf $\underline{\zvp}$ has trivial Tate support. In general, it is a closed union of strata (taken from any stratification along which $\F$ is smooth). Indeed, the Tate cohomology functors commute with $i_y^*$, so that $\supp_\T(\F)=\supp(T^0\F)\cup\supp(T^1\F)$. 
\end{Rem} Proposition \ref{suppprop} implies that this is a reasonable notion of support: 

\begin{Cor}For any complex $\F$ in $\mathrm{Sh}_c(Y;\T_\infty)$, let $i:\supp_\T(\F)\to Y$ be the closed embedding. The adjunction map
\begin{align*}
\F\to i_*i^*\F
\end{align*}
is an isomorphism.\end{Cor}
\begin{proof}The cone of the adjunction map is $j_!j^*\F$, where $j$ is the inclusion of the open set $Y\backslash \supp_\T(\F)$. The stalks of this are all $0$, so it is $0$ also.\end{proof}

In particular, any complex in $\mathrm{Sh}_c(Y;\T_\infty)$ is isomorphic in $\mathrm{Sh}_c(Y;\T_\infty)$ to one whose Tate support is equal to its usual support in the derived category. A nice corollary is that we may clear up the apparent disparity, mentioned in Remark \ref{broker?}, between our definition of the constructible Tate category and that of \cite{Tr}:

\begin{Cor}\label{fixer!}Let $\F$ be an object of $D^b_c(Y;\zvp)$ such that all the stalks of $\F$ are isomorphic to $0$ in $\T_\infty\modc$. Then $\F$ is a perfect complex.\end{Cor}

\begin{proof}Indeed, in that case $\bT^*\F$ has empty Tate support, so is $0$. But the kernel of $\bT^*$ is precisely the category of perfect complexes.\end{proof}

\subsection{Equivariance}If $G$ is an algebraic group acting on $Y$ and with $\vp\subset Z(G)$ lies in the center of $G$, one may analogously define $\mathrm{Sh}_{G,c}(Y;\T_\infty)$ as a localization of the bounded constructible equivariant derived category $D^b_{G,c}(Y;\bO)=D^b_{G,c}(Y;\zvp)$. One may also work with a fixed stratification $S$ which is compatible with the $G$-action, just as in the non-equivariant case. The six functor formalism also carries over, and we are able to define cohomological functors
\begin{align*}
T^0,T^1:\mathrm{Sh}_{G,c}(Y;\T_\infty)\to \mathrm{Sh}_{G,c}(Y;\ff)
\end{align*}
in essentially the same way. However, we are faced with certain technical difficulties in this situation which we do not yet know how to resolve; see Remark \ref{crazy}. For that reason, we will not pursue the $G$-equivariant situation.


\section{Tate-parity sheaves}\label{Section: Tate-parity}

We now develop a theory of parity sheaves in the integral Tate category. We continue with our assumption that $Y$ is a complex algebraic variety endowed with a trivial $\vp$-action. Additionally, we fix a stratification $S$ of $Y$ and assume that for each $\lam\in S$ the stratum $Y_\lam$ is simply connected. 

\subsection{Parity sheaves}\label{Section: party cond}
 We begin by recalling the classical notion from \cite{JMW1} and let $k$ be a complete local PID. Let $D^b_{S}(Y;k)$ denote the bounded $S$-constructible derived category. 
We fix a map of sets
\begin{align*}
\dagger:S\to\zz/2\zz,
\end{align*}
which is known as a \emph{pariversity}. While the notion of a parity sheaf depends on a choice of pariversity, we usually omit $\dagger$ from our notation. 

For $\lambda\in S$ we write $i_\lambda: Y_\lambda\to Y$ for the inclusion of the corresponding stratum in $Y$. 
\begin{Def}
A sheaf $\cF$ in $D^b_{S}(Y;k)$ is said to be ($\dagger$-) \emph{even} if the cohomology sheaves 
$H^n(i_\lambda^?\cF)$ have $k$-free stalks and vanish in degrees $n\in\dagger(\lambda)+1$. Here, $?$ is either $\ast$ or $!$. Likewise, $\cF$ is ($\dagger$-)\emph{odd} if $\cF[1]$ is ($\dagger$-) even.
\end{Def}
A sheaf $\cF$ is said to be a \emph{parity complex} if it is isomorphic to a direct sum of even and odd complexes. All of these properties are inherited by direct summands.

Recall that an additive category is \emph{Krull-Remak-Schmidt} if every object is isomorphic to a finite direct sum of objects, each of which has local endomorphism ring. In such a category, all idempotents split and any object admits a unique decomposition into indecomposable objects. Moreover, an object is indecomposable if and only if its endomorphism ring is local. Our assumptions on $k$ imply that the category $D^b_{S}(Y;k)$ is Krull-Remak-Schmidt. In particular, every parity complex is a direct sum of indecomposable parity complexes.

\subsection{Parity conditions}The theory of parity sheaves works best when $S$ is a JMW stratification, meaning a Whitney stratification such that the following condition holds:
\begin{Cond}\label{parcon}For any $\lambda\in S$, the cohomology 
\[
H^n(Y_\lam,k)=\Hom_{D^b_{S}(Y_\lambda;k)}(\underline{k},\underline{k}[n])
\]
is a free $k$-module for all integers $n$ and vanishes when $n$ is odd.

 Under this assumption, a complex $\cF$ in $D^b_{S}(Y;k)$ is even if and only if for each $\lam\in S$, $i_\lambda^?\cF$ is a direct sum of even shifts of $k$-free constant sheaves for both $?=*,!$ \cite[Remark 2.5]{JMW1}.\end{Cond}

The following result classifies indecomposable parity complexes:

\begin{Prop}\cite[Theorem 2.12]{JMW1}\label{Prop: parity uniqueness}
Let $\F$ be an indecomposable parity complex. Then
\begin{enumerate}
\item\label{xx1}
The support of $\F$ is irreducible, hence of the form $\overline{Y_\lam}$ for some $\lam\in S$.
\item\label{xx2}
The restriction $i_\lam^*\F$ is isomorphic to $\underline{k}[m]$, for some integer $m$.
\item\label{xx3}
Any indecomposable parity complex supported on $\overline{Y_\lam}$ and extending $\underline{k}[m]$ is isomorphic to $\F$.
\end{enumerate}
\end{Prop}

The indecomposable parity complex $\cF$ is called a \emph{parity sheaf} if $m=\dim_\cc(Y_\lam)=:d_\lambda$ in the above proposition. When such a complex exists, the proposition implies it is unique up to isomorphism and we denote it by $\mathcal{E}(\lam, k)$. The parity sheaf $\mathcal{E}(\lam, k)$ may or may not exist, depending on the situation; see \cite[Section 4]{JMW1}.

\subsection{Tate-parity complexes} We continue to assume our stratification $S$ of $Y$ satisfies Condition \ref{parcon}. Let $k=\bO.$ Our main definition is the following:

\begin{Def}\label{Def: Tate parity} Let $\F$ be an object in $\mathrm{Sh}_{S}(Y;\T_\infty)$. Let $?=*$ or $!$.
\begin{enumerate}
\item $\F$ is $?$-\emph{Tate-even} if for each $\lam\in S$,
\begin{align*}
T^1(i^{?}_\lam\F[\dagger(\lambda)]) = 0.
\end{align*}\
\item $\F$ is $?$-\emph{Tate-odd} if $\F[1]$ is $?$-Tate-even.\
\item $\F$ is \emph{Tate-even} (resp. \emph{Tate-odd}) if $\F$ is $?$-Tate-even (resp. -odd) for both $?=*$ and $?=!$.\
\item $\F$ is \emph{Tate-parity} if $\F$ is isomorphic in $\mathrm{Sh}_{S}(Y;\T_\infty)$ to the direct sum of a Tate-even and a Tate-odd object.\
\end{enumerate}
\end{Def}

The relationship between Tate-parity complexes and parity complexes is somewhat subtle. For instance, consider the case where $Y$ consists of a single point. Then any free $\bO$-module $M$ with trivial $\vp$-action is Tate-even, but $\zvp$-modules of the form $(\zvp/N)^n$ are Tate-odd, where $N=\sum_ig^i$ is the norm element.

\begin{Rem}Remark \ref{Rem: modular coefficients} shows that a non-zero complex of $\ff$-modules \emph{can} be Tate-parity, contrary to the fact that it can never be parity in $D^b_{S}(Y;\bO)$. Indeed,  Lemma \ref{allparity} implies that every object in \[
\mathrm{Sh}_{c}(\{pt\};\T_\infty)=\T_\infty\modc
\]
is Tate-parity. However,  a non-zero complex of $\ff$-modules will generally be neither Tate-even nor Tate-odd, so in particular will not be indecomposable.\end{Rem}

\subsection{Tate-parity sheaves}\label{Section: tateparitysheaves}Let  
\[
\mathrm{Par}_{S}(Y;\T_\infty)\subset\mathrm{Sh}_S(Y;\T_\infty)
\]denote the full additive subcategory spanned by Tate-parity complexes. We show below that this category is Krull-Remak-Schmidt in the sense mentioned in Section \ref{Section: party cond}. In particular, every Tate-parity complex is a sum of indecomposable Tate-parity complexes in a unique way. 

In this direction, let us assume that $Y$ consists of a single stratum. We have the following analogue of Proposition \ref{Prop: parity uniqueness}:

\begin{Thm}\label{megathm}Let $\F$ be a Tate-parity complex in $\mathrm{Sh}_S(Y;\T_\infty)$ and suppose $Y$ is simply connected. Then $\F\cong\underline{M}\;$ for some $M$ in $\T_\infty\modc$.\end{Thm}

\begin{proof}We may assume that $\cF$ is Tate-even. Recalling Corollary \ref{Prop: Grothendieck}, the claim is that there exists a $d$ such that $\F \cong \underline{\bO}^d$ in $\mathrm{Sh}_S(Y;\T_\infty)$. We show this using the Tate hypercohomology spectral sequence. Making use of the notations from Section \ref{hyperco}, we fix our dg-algebra $E$ in $D^b_c(\zvp)$. We note the similarity with the argument below and the proof of Lemma \ref{allparity}.

Since $\F$ is Tate-even, the Tate hypercohomology spectral sequence $E^{\bullet,\bullet}_2(\F)$ is concentrated in even bi-degrees, forcing it to collapse on this page. 
Additionally, the assumption that $Y$ satisfies the JWM condition \ref{parcon} implies that $E^{\bullet,\bullet}_2(\underline{\bO})$ collapses on the $E_2$-page. 
Given that $Y$ is simply connected, the $E^{\bullet,\bullet}_2(\underline{\bO})$-module $E^{\bullet,\bullet}_2(\F)$ is free. As both bi-gradings are concentrated in even bi-degrees, we may take generators in bi-degree $(0,0)$, obtaining an isomorphism
\begin{equation}\label{eqn: first iso}
    \mathrm{TH}^\bullet\F\cong (T^\bullet E)^d.
\end{equation}

First, we claim that this may be lifted to an isomorphism $\Pi_\ast\F\cong E^d$. Recall that in $\T_\infty\modc$, the functor $T^0$ is given by 
\[
T^0(M)=\Hom_{\T_\infty\modc}(\bT^\ast\bO,M),
\] for an object in $\T_\infty\modc$. By the tensor-restriction adjunction between $E\moddg/\Perf_{dg}(E)$ and $\T_\infty\modc$, we see that the functor 
\[
T^0:E\moddg/\Perf_{dg}(E)\lra T^0E\text{-}\mathrm{Mod}_{gr}
\]
is similarly given by $\Hom_{E\moddg/\Perf_{dg}(E)}(E,-)$. Therefore, choosing a basis of cardinality $d$ for $$\mathrm{TH}^0\F=T^0\Pi_*\F =\Hom_{E\moddg/\Perf_{dg}(E)}(E,\Pi_\ast\F)$$ gives a map
\begin{align}\label{second iso}
f:E^d\to \Pi_*\F
\end{align}
in $E\moddg/\Perf_{dg}(E)$ which induces the isomorphism (\ref{eqn: first iso}) after taking $T^\bullet$. Since the restriction functor to $\T_\infty\modc$ reflects isomorphisms (as it is triangulated and kills no objects) and commutes with $T^\bullet$, $f$ must be an isomorphism.

Finally, we claim that $\Pi_*$ is fully faithful. This implies it reflects isomorphisms, so that the isomorphism (\ref{second iso}) lifts to an isomorphism $\F \cong \underline{\bO}^d$ in $\mathrm{Sh}_S(Y;\T_\infty)$. As $\Pi_\ast$ is a triangulated functor, it is enough to check this on the generators $\underline{\bO}$ and $\underline{\bO}[1]$ of $\mathrm{Sh}_S(Y;\T_\infty)$, where we are using our assumption that $Y$ is simply connected. This calculation is essentially contained in what we have written before, combining (\ref{nice}) with Proposition \ref{moddyprop}. \end{proof}

\begin{Rem}\label{crazy}We suspect that if $Y$ is not simply connected, or in the $G$-equivariant setting, then it is true that any indecomposable Tate-parity complex is of the form $\bT\mathcal{L}$, up to shift, for some indecomposable $\bO$-free local system $\mathcal{L}$. 

The above proof does not work if $Y$ is not simply connected, because $\Pi_*$ is not fully faithful. We also have difficulties when $G$ is, say, an algebraic torus of positive dimension: if we try to run the above argument with $Y=BG$, we are immediately stuck because we cannot take bounded $\pi_*$-acyclic resolutions as was done in the construction of the Tate hypercohomology spectral sequence. Hopefully these issues are not essential.
\end{Rem}

\subsection{JMW redux}\label{Section: JWM redux}
We now return to the setting of a complex variety $Y$ equipped with a stratification $S$ satisfying Condition \ref{parcon}, and consider the trivial $\vp$-action on $Y$. The category $\mathrm{Par}_{S}(Y;\T_\infty)$ enjoys many of the same properties as the category of parity complexes. Using Theorem \ref{megathm}, many of the arguments are formally identical to those of \cite{JMW1}, and in these instances we omit the details.


\begin{Prop}\label{Prop: Hom-sets}
Suppose that $\F$ is $*$-Tate-parity and $\G$ is $!$-Tate-parity. Then there is a non-canonical isomorphism of $\ff$-vector spaces
\begin{align*}
\Hom_{\mathrm{Sh}_S(Y;\T_\infty)}(\F,\G) \cong \bigoplus_{\lam\in S}\Hom_{\mathrm{Sh}_S(Y;\T_\infty)}(i_\lam^\ast\F,i_\lam^!\G).
\end{align*}
In particular, if $\F$ is $*$-Tate-even and $\G$ is $!$-Tate-odd then $\Hom_{\mathrm{Sh}_S(Y;\T_\infty)}(\F,\G)=0$.
\end{Prop}
\begin{proof}
The proof is identical to that of \cite[Prop. 2.6]{JMW1}, using the calculations of Theorem \ref{megathm}.\end{proof}

We may now show that $\mathrm{Par}_{S}(Y;\T_\infty)$ is Krull-Remak-Schmidt.

\begin{Prop}\label{prop: krull} The categories of $*$-Tate-even complexes and $!$-Tate-even complexes are both Krull-Remak-Schmidt. In particular, the category of Tate-parity complexes $\mathrm{Par}_{S}(Y;\T_\infty)$ is Krull-Remak-Schmidt.\end{Prop}

\begin{proof}It suffices to show that a $*$-Tate-even complex $\F$ in $\mathrm{Sh}_S(Y;\T_\infty)$ splits as a direct sum of objects of $\mathrm{Sh}_S(Y;\T_\infty)$, each of which has a local endomorphism ring. Any such summand will automatically be $*$-Tate-even; the same holds for Tate-odd. 

Since endomorphism rings of objects of $\mathrm{Sh}_S(Y;\T_\infty)$ are all finite-dimensional $\ff$-vector spaces, it is enough to show that any idempotent endomorphism of a $*$-Tate-even complex is split. In \cite{LC}, it is shown that in any triangulated category an idempotent endomorphism of a distinguished triangle which splits on any two terms splits on the third. For any stratum $Y_\lam$ open in the Tate support of $\F$, if we set $j_\lam$ to be the inclusion and $i_\lam$ to be the inclusion of the compliment, we may use the functorial distinguished triangle
\begin{align*}
j_{\lam,!}j^!_\lam\F\to \F\to i_{\lam,*}i^*_\lam\F\xrightarrow{+1}
\end{align*}
which preserves the $*$-Tate-even property, to reduce by induction on the dimension of $Y$ to the case of a single connected stratum. By Theorem \ref{megathm}, a Tate-even complex on a single stratum $Y$ is a direct sum of copies of $\underline{\bO}$, whose endomorphism algebra is the local ring $H^\bullet(Y;\ff)$.

Arguing in a similar fashion with the dual functorial distinguished triangles (preserving the $!$-Tate-even property), any idempotent endomorphism of a $!$-Tate-even complex is also split. Since there are no maps between Tate-even and Tate-odd complexes by Proposition \ref{Prop: Hom-sets}, the splitting of idempotent endomorphisms of Tate-parity complexes follows from the Tate-even case, which follows in turn from the $*$-Tate even case; that summands of Tate-even complexes are Tate-even is automatic.\end{proof}

Before stating the application of this property to the classification of indecomposable Tate-parity complexes, we collect a few immediate consequences.
\begin{Cor}The category of Tate-parity complexes is closed under $\mathbb{D}$.\end{Cor}
\begin{proof}Theorem \ref{megathm} shows that $\mathbb{D}$ exchanges $*$-Tate-even and $!$-Tate-even complexes, so sends Tate-even complexes to Tate-even complexes (as required). \end{proof}

 \begin{Cor}\label{Cor: surj}
 Suppose $\F$ and $\G$ are indecomposable Tate-parity complexes of the same parity and let $j: Y_\mu\hra Y$ be the inclusion of a stratum $Y_\mu$ which is open in the support of both $\F$ and $\G$. Then 
 \begin{align*}
 \Hom_{\mathrm{Sh}_S(Y;\T_\infty)}(\F,\G) \twoheadrightarrow \Hom_{\mathrm{Sh}_S(Y;\T_\infty)}(j^\ast\F,j^\ast\G)
 \end{align*}
 is a surjection.
 \end{Cor}
 \begin{proof}
 Let $V = \supp_\T(\F)\cup\supp_\T(\G)$, and let $i: V\setminus{Y_\mu}\hra Y$ be the inclusion of the closed complement. Applying the functor $\Hom_{\mathrm{Sh}_S(Y;\T_\infty)}(\F,-)$ to the triangle 
 \begin{align*}
 i_\ast i^!\G \lra \G\lra j_\ast j^\ast\G\xrightarrow{+1}
 \end{align*}
 and studying the corresponding long exact sequence, we see that the result follows if \begin{align*}\Hom_{\mathrm{Sh}_S(Y;\T_\infty)}(\F,i_\ast i^!\G[1])\cong \Hom_{\mathrm{Sh}_S(Y;\T_\infty)}(i^\ast\F,i^!\G[1])\end{align*} vanishes. But this follows from Proposition \ref{Prop: Hom-sets}.
 \end{proof}
 
\begin{Cor}\label{Prop: indecomp to indecomp}
Suppose $J: U\lra Y$ is the inclusion of an open union of strata. Then for an indecomposable Tate-parity complex $\F$, the restriction of $\F$ to $U$ is either $0$ or indecomposable.
\end{Cor}
 \begin{proof}
 This follows, as in \cite[Prop 2.11]{JMW1}, from Corollary \ref{Cor: surj} and the fact that the category of Tate-parity complexes is Krull-Remak-Schmidt. 
 \end{proof}
The following is the analogue of Proposition \ref{Prop: parity uniqueness} for Tate-parity complexes, which follows readily from the results in this section.
 \begin{Thm}\label{Thm: Tate uniqueness}
 Let $\F$ be an indecomposable Tate-parity complex.
 \begin{enumerate}
 \item The support of $\F$ is of the form $\overline{Y_\lam}$ for a unique stratum $Y_\lam\subset Y$.
 \item Suppose $\G$ and $\F$ are two indecomposable Tate-parity complexes such that 
 \begin{align*}
 {\supp_\T(\G) =\supp_\T(\F)}.
 \end{align*}
 Let $j:Y_\lam\hra Y$ be the inclusion of the unique stratum open in this support. If $j^\ast \G \cong j^\ast \F$, then $\G\cong \F$.
 \end{enumerate}
 \end{Thm}
 \begin{proof}Given the preceding results, the argument is now identical to the proof of \cite[Thm 2.12]{JMW1}.\end{proof} 

This motivates the following definition.


\begin{Def} A complex $\F$ in $\mathrm{Par}_S(Y;\T_\infty)$ is a \emph{Tate-parity sheaf} if it is an indecomposable Tate-parity complex and its restriction to the unique stratum $Y_\lambda$ which is dense in its support is of the form $\bT (\underline{\bO}[d_\lambda])$ (see (\ref{eqn: tate base change})) where $d_\lambda$ is the complex dimension of $Y_\lam$. 
\end{Def}

The previous theorem implies that if such an $\cF$ exists then it is unique up to isomorphism, and will be denoted $\mathcal{E}_\T(\lambda)$.

\subsection{Modular reduction}\label{Section: modular reduction}
 The behavior of parity sheaves under change of base functors is studied in \cite[Section 2]{JMW1}. In the case of modular reduction $\bO\to\ff$, the following results are established.
 \begin{Prop}\label{Prop: base change}
 Let $?=*,!$. A complex $\F\in D^b_S(X;\bO)$ is $?$-even (resp. $?$-odd) if and only if $\ff(\F)\in D^b_S(X;\ff)$ is $?$-even (resp. $?$-odd).
 \end{Prop}
This leads to a correspondence between parity sheaves with $\bO$- and $\ff$-coefficients:
 \begin{Prop}\label{Prop: mod p}
 If $\mathcal{E}\in D^b(X;\bO)$ is a parity sheaf, then $\ff\mathcal{E}$ is a parity sheaf. In fact,
 \[
\ff \mathcal{E}(\lam, \bO) \cong \mathcal{E}(\lam, \ff).
 \]
 \end{Prop}
We now prove analogues of Propositions \ref{Prop: base change} and \ref{Prop: mod p} for Tate-parity sheaves. In this setting, the relevant functor is the {Tate extension of scalars functor} $\bT=\bT^* \epsilon^*$ from Section \ref{Sec: scalars}.

 \begin{Prop}\label{Prop: Tate base change}
Let $?=*,!$. If $\F$ is a $?$-even (resp. $?$-odd) complex in $D^b_S(Y;\bO)$, then $\bT\F$ is $?$-Tate-even (resp. $?$-Tate-odd).
 \end{Prop}
\begin{proof}
Since $\bT$ is compatible with $i_\lambda^?$, a standard d\'{e}vissage argument reduces the statement to a single stratum. In this case, we have
\begin{align*}
T^i\bT\underline{\bO}=\left\{\begin{matrix}
\mathbb{F} & i=0 \\ 
 0& i=1 
\end{matrix}\right.
\end{align*}
by the proof of Lemma \ref{Lem: valued in Fp mod}.\end{proof}

\begin{Thm}\label{sheavesbij}Suppose the parity sheaf $\mathcal{E}=\mathcal{E}(\lambda,\bO)$ exists and satisfies 
\begin{align*}
\Hom_{D^b_c(Y;\bO)}(\mathcal{E},\mathcal{E}[n])=0
\end{align*}
for all $n<0$ (for instance, $\mathcal{E}$ may be perverse). Then:
\begin{align*}
\bT\mathcal{E}=\mathcal{E}_{\T}(\lambda).
\end{align*} \end{Thm}
\begin{proof} As $\bT$ is compatible with $i_\lambda^\ast$, it is enough to show that $\bT\mathcal{E}$ is indecomposable by Theorem \ref{Thm: Tate uniqueness}. Since $\bT\mathcal{E}$ is Tate-parity, we need only show its endomorphism algebra is local. 

By Proposition \ref{moddyprop}, we have
\begin{align*}
\Hom_{\mathrm{Sh}_c(Y;\T_\infty)}(\bT\mathcal{E},\bT\mathcal{E})=\bigoplus_{i\in\zz}\Hom_{D^b_c(Y;\ff)}(\bF\mathcal{E},\bF\mathcal{E}[2i]).
\end{align*}
Since $\mathcal{E}$ is parity, the RHS is equal to 
\begin{align*}
\bigoplus_{i\in\zz}\bF \Hom_{D^b_c(Y;\bO)}(\mathcal{E},\mathcal{E}[2i])=\bigoplus_{i\geq 0}\bF \Hom_{D^b_c(Y;\bO)}(\mathcal{E},\mathcal{E}[2i])
\end{align*}
which is a non-negatively graded algebra whose degree $0$ subalgebra is
\begin{align*}
\bF \Hom_{D^b_c(Y;\bO)}(\mathcal{E},\mathcal{E}).
\end{align*}
This is local, being the quotient of a local ring. But any finite-dimensional non-negatively graded algebra is local if and only if its degree $0$ subalgebra is.\end{proof}

The following is immediate:

\begin{Cor}\label{superbij}Suppose all parity sheaves $\mathcal{E}(\lambda,\bO)$ exist and satisfy 
\begin{align*}
\Hom_{D^b_c(Y;\bO)}(\mathcal{E}(\lambda,\bO),\mathcal{E}(\lambda,\bO)[n])=0
\end{align*}
for all $n<0$ (this holds, for instance, if all the parity sheaves are perverse). Then all Tate-parity sheaves exist and $\bT$ induces a bijection between parity sheaves and Tate-parity sheaves.\end{Cor}

For the remainder of the section, we assume we are in the setting of Corollary \ref{superbij}. It follows that every Tate-parity complex may be uniquely written as a direct sum of objects of the form $\mathcal{E}_{\T}(\lambda)$ and $\mathcal{E}_{\T}(\lambda)[1]$.

For simplicity, let us assume that $d_\lam=\dim_\cc(Y_\lam)$ is even for each $\lam\in S$. Recalling Proposition \ref{moddyprop} and Section \ref{addendum}, we have a factorization
\begin{align*}
\ff: \mathrm{Par}_{S}^0(Y;\bO)\xrightarrow{\bT}\mathrm{Par}_{S}^0(Y;\T_\infty)\xrightarrow{L}\mathrm{Par}_{S}^0(Y;\ff),
\end{align*}
where $\mathrm{Par}_{S}^0(Y;-)$ stands for the category of \emph{normal} (Tate-) parity complexes. Here a (Tate-) parity complex is said to be normal if it is a direct sum of (Tate-) parity sheaves without any shifts. More precisely, for parity sheaves $\cE,\cE'$ we have 
\begin{align*}
\Hom_{\mathrm{Sh}_S(Y;\T_\infty)}(\bT\cE, \bT\cE')=\bigoplus_{i\geq0}\Hom_{D^b_c(Y;\ff)}(\bF\cE,\bF\cE'[2i])
\end{align*}
which admits as a quotient $\Hom_{D^b_c(Y;\ff)}(\bF\cE,\bF\cE')$. The quotient map is compatible with compositions, so that we have a well-defined functor $L$ which takes the object 
\begin{align*}
\cE_{\T}(\lambda,\ff)=\bT\cE(\lambda,\bO)
\end{align*}
to 
\begin{align*}
\cE(\lambda,\bF)=\bF\cE(\lambda,\bO).
\end{align*}
We record this as:

\begin{Thm}\label{lifting}Suppose all parity sheaves $\mathcal{E}(\lambda,\bO)$ exist and satisfy 
\begin{align*}
\Hom_{D^b_c(Y;\bO)}(\mathcal{E}(\lambda,\bO),\mathcal{E}(\lambda',\bO)[n])=0
\end{align*}
for all $n<0$. Then there is a functor $L:\mathrm{Par}_{S}^0(Y;\T_\infty)\to \mathrm{Par}_{S}^0(Y;\ff)$ such that the composition
\begin{align*}
\mathrm{Par}_{S}^0(Y;\bO)\xrightarrow{\bT}\mathrm{Par}_{S}^0(Y;\T_\infty)\xrightarrow{L}\mathrm{Par}_{S}^0(Y;\ff)
\end{align*}
is equal to the modular reduction functor $\bF$.\end{Thm}

\subsection{A technical remark}It may happen that there exists a normalization of the indecomposable parity complexes different from the normalization by parity sheaves which still satisfies the condition of Theorem \ref{sheavesbij}. Additionally, when the assumption that the strata are even dimensional is dropped, certain shifts are necessary. In such cases, the results of the previous section hold just the same, since we never used the value $d_\lambda$ in any technical way - it is there simply so that the phrase ``for instance, if all the parity sheaves are perverse" makes sense. 

More generally, write $n$ for the difference between this normalization and the normalization by parity sheaves, and write $\mathrm{Par}_{S}^n(Y;\bO)$, $\mathrm{Par}_{S}^n(Y;\ff)$, $\mathrm{Par}_{S}^n(Y;\T_\infty)$ for the corresponding categories. For simplicity of exposition, we will consider only $n=0$ (that is, the ``normal'' case), but the reader may trivially extend our results to the case of general $n$. We will return to this notation in Section \ref{liftingsection}.

\section{Smith theory and parity sheaves}\label{Section: Smith}

Now consider a complex algebraic variety $X$ with an action of $\vp$, and write $i:X^\vp\to X$ for the inclusion of the fixed-point subvariety. In this section, we develop the basics of Treumann's \emph{Smith theory for sheaves} in the integral context and study its compatibility with parity objects.

We consider the two functors 
\begin{align*}
i^!,i^*:D^b_{\vp,c}(X;\bO)\to D^b_{\vp,c}(X^\vp;\bO).
\end{align*}
We have:

\begin{Lem}\label{Lem: cone perfect}The cone of the natural map
\begin{align*}
i^!\to i^*
\end{align*}
is contained in $\mathrm{Perf}_c(X^\vp;\bO)$.\end{Lem}

\begin{proof}
The proof is exactly the same as the one given in \cite[Theorem 4.7]{Tr} in the $\ff$-case. For the convenience of the reader, we sketch the argument here. 

For any $\F$ in $D^b_{\vp,c}(X;\bO)$, set 
\[
C(\F)=\mathrm{cone}\left(i^!\F\lra i^\ast\F\right).
\]
Corollary \ref{fixer!} implies that it suffices to show that the stalks of $C(\F)$ all lie in $\Perf(\zvp)$. Let $x\in X^\vp$ be a fixed point, and let $U$ be a regular neighborhood of $x$, chosen to be $\vp$-invariant. Set $L= U-(U\cap X^\vp)$. Then $i^\ast_x(C(\F))$ is quasi-isomorphic to the the complex $\pi_\ast(\F|_L)$ computing the cohomology of $L$ with values in $\F|_L$. Lemma 4.1 in \cite{Tr} now implies that the freeness of the $\vp$-action on $L$ forces this complex to lie in $\Perf(\zvp).$
\end{proof}

It follows that the two functors 
\begin{align*}
\bT^*i^!,\bT^*i^*:D^b_{\vp,c}(X;\bO)\to \mathrm{Sh}_c(X^\vp;\T_\infty)
\end{align*}
are naturally isomorphic. We call `the' resulting functor the \emph{Smith functor}:
\begin{align*}
\Psm:D^b_{\vp,c}(X;\bO)\xrightarrow{i^\ast}D^b_{\vp,c}(X^\vp;\bO)\cong D^b_{c}(X^\vp;\zvp)\xrightarrow{\bT^\ast} \mathrm{Sh}_c(X^\vp;\T_\infty).
\end{align*}
By functoriality, it induces maps 
\begin{align*}
\Psm: \bigoplus_{k\geq0}\Hom_{D^b_{\vp,c}(X;\bO)}(\F,\G[2i])\to \Hom_{\mathrm{Sh}_c(X^\vp;\T_\infty)}(\Psm(\F),\Psm(\G))
\end{align*}
compatible with compositions. In particular, when $\F=\G$ one obtains a map of rings from the `even extension' algebra to the endomorphism algebra of $\Psm(\F)$.

\begin{Lem}\label{onto!}The map above is surjective.\end{Lem}

\begin{proof}By Proposition \ref{colimy} and Lemma \ref{Lem: cone perfect}, every map $a$ in $\mathrm{Sh}_c(X^\vp;\T_\infty)$ from $\Psm(\F)$ to $$\Psm(\G)= \bT^*i^\ast\G\cong  \bT^*i^!\G$$ is of the form 
\begin{align*}
\bT^*i^*\F\xrightarrow{\bT^*b} \bT^*i^!\G[2k]\cong \bT^*i^!\G
\end{align*}
for some morphism $b:i^*\F\to i^!\G[2k]$ in $D^b_{\vp,c}(X^\vp;\bO)$. The morphism $b$ determines by the adjunction between $i^*,i_*$ a map
\begin{align*}
\F\to i_*i^!\G[2k]= i_!i^!\G[2k]
\end{align*} 
in $D^b_{\vp,c}(X;\bO)$, which determines by the adjunction between $i_!,i^!$ a map
\begin{align*}
c:\F\to \G[2k]
\end{align*}
in $D^b_{\vp,c}(X;\bO)$. We leave it to the reader to check that $\Psm (c)=a$. \end{proof}

\subsection{Relation to parity sheaves}

Under certain circumstances which occur frequently in practice, the Smith functor $\Psm$ is well adapted to study the relationship between parity complexes on $X$ and Tate-parity complexes on $X^\vp$. We fix a $\vp$-invariant stratification $S$ of $X$. This induces in the natural manner a stratification $S^\vp$ of $X^\vp$. We further assume that both stratifications satisfy the JMW criterion (\ref{parcon}). The arguments of this section do not require the strata of $X^\vp$ to be simply connected, so we do not impose this condition.

An object of the equivariant derived category $D^b_{\vp,S}(X,\bO)$ is said to be even, odd, or parity if its underlying object in $D^b_{S}(X,\bO)$ is. The following is one of our main results:

\begin{Thm}\label{alloallo}
Let $\cE$ be a parity complex in $D_{\vp,S}^b(X;\bO)$ such that for each $\vp$-fixed point $i_x:\{x\}\to X^\vp$ and every $k\in \zz$, the cohomology groups $H^k(i_x^?\cE)$ is trivial as a $\vp$-module for both $?=!,*$. Then $\Psm(\cE)$ is Tate-parity.
\end{Thm}
\begin{proof}
We will show that if $\cE$ is even then $\Psm(\cE)$ is Tate-even. 
For any stratum $X_\lambda$, set $X^\vp_\lam = X^\vp\cap X_\lam$ and write 
\[
i_\lambda:X_\lambda\to X,\quad i_\lambda^\vp:X^\vp_\lambda\to X^\vp
\]for the locally-closed embeddings and $i:X^\vp\to X$. For either $?=*$ or $!$, we have
\begin{align*}
(i_\lambda^\vp)^?\bT^*i^?
=\bT^*(i_\lambda^\vp)^?i^?
=\bT^*(i^\lambda)^?(i_\lambda)^?.
\end{align*}
where $i^\lambda:X^\vp_\lambda\to X_\lambda$ denotes the (regular) embedding. Setting $\Psm_?:=\bT^\ast i^?,$ Lemma \ref{Lem: cone perfect} gives a natural isomorphism $\Psm_!\xrightarrow{\sim}\Psm_\ast$ and the above computation gives
\[
(i_\lambda^\vp)^?\Psm_?=\bT^*(i^\lambda)^?(i_\lambda)^?.
\]
It thus suffices to show that $\bT^*(i^\lambda)^?(i_\lambda)^?$ is Tate-even for all $\lam$ and for both $?=!,\ast$.

For $\cE$ as in the statement of the theorem, $(i_\lambda)^?\cE$ is an even complex on the single stratum $X_\lambda$; without loss of generality we assume $\dagger(\lambda)=0$. We are given that:
\begin{enumerate}
\item $H^k((i_\lambda)^?\cE)$ is a $\bO$-free $\vp$-equivariant local system which vanishes for odd $k$; and\
\item the cohomology modules of the $?$-restriction of $(i_\lambda)^?\cE$ on any point of $X^\vp_\lambda$ are trivial $\vp$-modules.
\end{enumerate}
It is clear that these properties also hold for $(i^\lambda)^\ast(i_\lambda)^\ast\cE$. For complexes constructible along $X_\lambda$, we have a Gysin isomorphism between $(i^\lambda)^*$ and $(i^\lambda)^!$ up to a shift by twice the codimension. In particular, $(i^\lambda)^!(i_\lambda)^!\cE$ also satisfies these two properties. 

Similarly using the Gysin isomorphism arising from the inclusion of any point $x$ in $X^\vp_\lambda$, we conclude that the cohomology sheaves of $(i^\lambda)^?(i_\lambda)^?\cE$ are local systems of free $\bO$-modules with the \emph{trivial} $\vp$-action. By the parity conditions on the stratification, all odd extensions in $D^b_c(X^\vp_\lambda,\zvp)$ between such local systems vanish, so that $(i^\lambda)^?(i_\lambda)^?\cE$ splits as a direct sum of its cohomology sheaves. Applying $\bT^*$ to this evidently gives a Tate-even complex.\end{proof}

The assumptions of the theorem frequently occur in practice. The main situation we have in mind is the following:

\begin{Lem}\label{s1lem}Suppose the $\vp$-action on $X$ extends to an action of $S^1$ such that every connected component of every stratum $X^\vp_\lambda$ contains an $S^1$-fixed point. Then the conditions of Theorem \ref{alloallo} are satisfied whenever the $\vp$-equivariant structure  of $\cE$ comes from an $S^1$-equivariant structure.\end{Lem}

\begin{proof}Let ${X^\vp_\lambda}'$ be a connected component of $X^\vp_\lambda$. The cohomology sheaves of the restriction of $\cE$ to ${X^\vp_\lambda}'$ are local systems of $\zvp$-modules with at least one stalk for which the action of $\vp$ extends to one of $S^1$, and therefore must be trivial. Since ${X^\vp_\lambda}'$ is connected the action of $\vp$ must be trivial on every stalk, which proves the case $?=*$.

For the case $?=!$, we note that $\mathbb{D}\cE$ is also $S^1$-equivariant, so the cohomology modules of its stalks on points of $Y$ again all have the trivial $\vp$-action. Since $\mathbb{D}\cE$ is parity, it follows that these stalks are all isomorphic to their cohomology. Therefore, their Verdier duals, namely the costalks of $\cE$, have the required property.\end{proof}

\subsection{Lifting}\label{liftingsection}
We now combine the results of the previous section and the results of Section \ref{Section: modular reduction}. Suppose that $X$ and $X^\vp$ are as in the previous section. Let $\cE$ be a parity complex in $D^b_{\vp,S}(X;\bO)$ satisfying the conditions of Theorem \ref{alloallo}, so that $\Psm(\cE)$ is a Tate-parity complex. Suppose further that all parity sheaves $\cE(\lambda,\bO)$ exist for $D^b_{S}(X^\vp;\bO)$, and moreover satisfy the conditions of Theorem \ref{sheavesbij}. Corollary \ref{superbij} now implies that $\bT$ induces a bijection
\[
\bT: \{\text{$\bO$-parity sheaves on $X^\vp$}\}\lra \{\text{Tate-parity sheaves on $X^\vp$}\}.
\]In particular, there is a unique parity complex $\F$, up to even shifts in the summands, satisfying the equation
\begin{align*}
\Psm(\cE)\cong\bT\F.
\end{align*} 
Let us denote by $\mathrm{Par}^n_{\vp,S}(X;\bO)$ the full subcategory of $D^b_{\vp,S}(X;\bO)$ spanned by complexes satisfying the conditions of Theorem \ref{alloallo} and whose image under $\Psm$ is contained in $\mathrm{Par}_{S}^0(X^\vp;\T_\infty)$. We may then consider the composition
\begin{align*}
LL:\mathrm{Par}^n_{\vp,S}(X;\bO)&\xrightarrow{\Psm}\mathrm{Par}_{S}^0(X^\vp;\T_\infty)\xrightarrow{L}\mathrm{Par}_{S}^0(X^\vp;\ff)\\
\cE\qquad&\longmapsto \Psm(\cE)\cong\bT\F\longmapsto \quad\ff(\F),
\end{align*}
where $L$ is the functor from Theorem \ref{lifting}. In particular, we obtain a functor between two full subcategories of derived categories, constructed by passing through the Tate category.

\section{A geometric construction of the Frobenius contraction functor}\label{Section: geometric application}
In our final section, we give an application of the above theory to geometric representation theory. We begin by recalling the setting of the geometric Satake equivalence and the related results about parity sheaves. Recalling the Frobenius contraction functor of Gros and Kaneda, we combine our techniques with Treumann's modular Tate category (\cite{Tr}) to give a geometric construction of this functor on the category of spherical perverse sheaves.
\subsection{Geometric Satake equivalence}\label{tiltsec}Let $G$ be a complex algebraic connected reductive group, and let $\calGr_G$ be the affine Grassmannian of $G$. It is an ind-algebraic variety, which admits an action of a certain pro-algebraic group $G(\mathcal{O})$. As a set, $\calGr_G$ is the coset space $G(\mathcal{K})/G(\mathcal{O})$, where $\mathcal{K}=\mathbb{C}((t))$, $\mathcal{O}=\mathbb{C}[[t]]$. The action of $G(\mathcal{O})$ factors locally through an algebraic quotient group, so that its orbits are all simply-connected algebraic subvarieties of $\calGr_G$. Let us fix a maximal torus $T$ of $G$. It is known that the $T$-fixed point set in $\calGr_G$ is equal to the cocharacter lattice $\mathbb{X}_\bullet(T) = \calGr_T$ of $T$. Every $G(\mathcal{O})$-orbit on $\calGr_G$ contains a unique Weyl group-orbit in $\mathbb{X}_\bullet(T)$. If we fix further a Borel subgroup containing $T$, then each such Weyl group-orbit contains a unique dominant cocharacter, so that the $G(\mathcal{O})$-orbits are in bijection with the set $\mathbb{X}_\bullet(T)^+$ of dominant cocharacters:
\begin{align*}
\calGr_G= \bigsqcup_{\lam\in \mathbb{X}_\bullet(T)^+}\calGr^\lam.
\end{align*}
This forms a Whitney stratification, called the \emph{spherical} stratification and denoted $sph$. Note also that $\dim_\cc(\calGr^\lam) = \la2\rho,\lam\ra$. The action of $\mathbb{G}_m$ on $\mathcal{O}$ induces its so-called `loop rotation' action on $\calGr_G$; this action is compatible with the action of $G(\mathcal{O})$, so that the orbits of $G(\mathcal{O})\rtimes\mathbb{G}_m$ on $\calGr_G$ coincide with the spherical strata.

Let $k$ be a Noetherian ring of finite homological dimension. The category 
\begin{align*}
D^b_{G(\mathcal{O})\rtimes\mathbb{G}_m}(\calGr_G;k)
\end{align*}
is equipped with a (convolution) monoidal structure. The geometric Satake equivalence (see \cite{MV}) asserts that its subcategory $\Perv_{G(\mathcal{O})\rtimes\mathbb{G}_m}(\calGr_G;k)$ of perverse sheaves is a monoidal subcategory, is equivalent\footnote{Under the map which forgets the equivariant structure.} to $\Perv_{sph}(\calGr_G;k)$, and that the monoidal functor
\begin{align*}
\Perv_{G(\mathcal{O})\rtimes\mathbb{G}_m}(\calGr_G;k)\to \Perv_{sph}(\calGr_G;k)\xrightarrow{H^*(\calGr_G,-)}k\text{-}\mathrm{Mod}
\end{align*}
integrates to a monoidal equivalence 
\[
\Perv_{G(\mathcal{O})\rtimes\mathbb{G}_m}(\calGr_G;k)\cong \mathrm{Rep}(G_{k}^\vee).
\]
Here $G_{k}^\vee$ denotes the split connected reductive group over $k$ which is Langlands dual to $G$.

Fix a regular dominant cocharacter $\mu$, and let $R^G_T:D^b_{sph}(\calGr_G)\lra D^b_{sph}(\calGr_T)$ denote the corresponding hyperbolic localization functor (see \cite{Br}). Then there is a commutative diagram 
\begin{equation}\label{diagram: restriction}
\begin{tikzcd}
\Perv_{sph}(\calGr_G;k)\ar[r,"\sim"]\ar[d,"R^G_T"]& \mathrm{Rep}(G_{k}^\vee)\ar[d,"Res_{T_{k}^\vee}^{G_{k}^\vee}"]\\
\Perv_{sph}(\calGr_T;k)\ar[r,"\sim"]&\mathrm{Rep}(T_{k}^\vee),
\end{tikzcd}
\end{equation}
 with the caveat that the $\lambda$-weight space will be placed in homological degree $-\langle2\rho,\lambda\rangle$. We remark that the hyperbolic localization is a monoidal functor.
 
 Now assume that $p$ is a good prime for $G$. It is shown in \cite{MR} (extending the original weaker result of \cite{JMW2}) that, when $k=\bO$ or $\ff$, all spherical parity sheaves with respect to the dimension pariversity exist and are perverse, and moreover that for $k=\ff$ the geometric Satake equivalence induces an equivalence
\begin{align*}
\mathrm{Par}^0_{sph}(\calGr_G;\ff)\cong \mathrm{Tilt}(G_{\ff}^\vee),
\end{align*}
where $\mathrm{Par}^0_{sph}(\calGr_G;\ff)$ is the additive category of spherical parity sheaves and the right-hand side denotes the category of tilting modules. We refer the reader to \emph{loc.cit.} for the definition and properties of $\mathrm{Tilt}(G_{\ff}^\vee).$ 
\subsection{The Frobenius contraction functor}
 The category $\mathrm{Rep}_{}(G_{\ff}^\vee)$ has several interesting constructions unique to the modular setting which may often be fruitfully studied via the topology of $\calGr_G$ via the geometric Satake equivalence. An important example is the Frobenius twist functor $\F r$, which takes the $\lam$-weight space of a representation $V$ to the $p\lam$-weight space of $\F r(V)$. Our goal is to show that the Smith functor gives a geometric realization of a right adjoint of a certain $\rho$-shifted version of $\F r(V)$. 
 
 Recently, Gros and Kaneda (\cite{GK0}, \cite{GK}) defined and studied a functor 
 \[
 \mathcal{F}r^{(-1)}:\mathrm{Rep}(G_{\ff}^\vee)\lra \mathrm{Rep}(G_{\ff}^\vee),
 \]called the Frobenius contraction functor; it is a right adjoint functor to the functor $St^{\otimes 2}\otimes\F r(-)$. Here $St=L((p-1)\rho)$ is the Steinberg module, which is irreducible of highest weight $(p-1)\rho$. On an object $V$ in $\mathrm{Rep}_{}(G_{\ff}^\vee)$, the representation $\F r^{(-1)}V$ is determined as a $T_{\ff}^\vee$-module by 
 \[
 \text{$(\F r^{(-1)}V)_\lam = V_{p\lam}$ for each $\lam\in \mathbb{X}^\bullet(T^\vee)$.}
 \]
 The existence and properties of $\F r^{(-1)}$ has several applications to linkage and Frobenius splittings (\cite{A}, \cite{GK0}). Importantly for us, it is known that $\F r^{(-1)}$ takes tilting modules to tilting modules (see \cite[Theorem 3.1]{GK}; also \cite[Corollary 3.7]{A}), so that we obtain a functor
\[
\F r^{(-1)}: \mathrm{Tilt}(G_{\ff}^\vee)\lra \mathrm{Tilt}(G_{\ff}^\vee).
\] 
\subsection{An extension of the Smith functor}\label{Section: Tate perverse}
Our present goal is to use Smith theory to construct an endofunctor 
\[
\Perv_{sph}(\calGr_G;\ff)\lra \Perv_{sph}(\calGr_G;\ff)
\]
which will give a geometric construction of the Frobenius contraction functor of Gros-Kaneda as in Theorem \ref{Thm: intro motivation}. As a first step, we construct a refinement of the localization functor $LL$ from Section \ref{liftingsection} in our present context. 

The equivalence 
\begin{align*}
\Perv_{G(\mathcal{O})\rtimes\mathbb{G}_m}(\calGr_G;\bO) \cong~\Perv_{sph}(\calGr_G;\bO)
\end{align*}
factors through an equivalence
\begin{align*}
\Perv_{\mathbb{G}_m, sph}(\calGr_G;\bO)\cong \Perv_{sph}(\calGr_G;\bO).
\end{align*}
Consider the composition, denoted $M$:
\begin{align*}
\mathrm{Par}^0_{sph}(\calGr_G;\bO)\xhookrightarrow{} \Perv_{sph}(\calGr_G;\bO)\cong \Perv_{\mathbb{G}_m, sph}(\calGr_G;\bO)\to \Perv_{\vp, sph}(\calGr_G;\bO).
\end{align*}
Here the last arrow is the restriction of equivariance to the subgroup $\vp\cong \mu_p\subset\mathbb{G}_m$ of $p^{th}$ roots of unity. 

The fixed point set $\calGr_G^\vp$ is a disjoint union of partial affine flag varieties (see \cite[1.3]{RicheWilliamsonSmith}). Noting that none of the theory developed in this paper is affected by restricting attention to a component of the fixed point variety, we restrict our attention to the component containing $p\mathbb{X}_\bullet(T)$. This component, which we denote by ${}^p\calGr_G$, is easily seen to be isomorphic to $\calGr_G$, embedded in the following way:
\begin{align*}
{}^p\calGr_G= G(\mathbb{C}((t^p)))/G(\mathbb{C}[[t^p]])\xhookrightarrow{}G(\mathbb{C}((t)))/G(\mathbb{C}[[t]])\cong \calGr_G.
\end{align*}
Note that the stratification of the embedded copy of $\calGr_G$ induced by $sph$ is again $sph$, and that for any stratum $Gr^\lam$ of $sph$, its intersection with the embedded copy of $\calGr_G$ is either empty or $\lam = p\mu$ for some cocharacter  $\mu$ and the intersection ${}^p\calGr_G\cap Gr^{\lam} \cong Gr^\mu$ is connected of $1/p^{th}$ the dimension.

\begin{Lem}
The composition $\Psm\circ M$ has image contained in $\mathrm{Par}^0_{ sph}({}^p\calGr_G;\T_\infty)$.
\end{Lem} 
\begin{proof}
The functor $M$ does not change the underlying complex, so it is still parity; since $M$ passes through a restriction of equivariance from $\mathbb{G}_m$ to $\vp$, Lemma \ref{s1lem} implies that the conditions of Proposition \ref{alloallo} hold. It remains to show that the induced functor
\begin{align*}
\mathrm{Par}^0_{sph}(\calGr_G;\bO)\lra \mathrm{Par}_{sph}({}^p\calGr_G;\T_\infty)
\end{align*}
lands inside $\mathrm{Par}_{sph}^0({}^p\calGr_G;\T_\infty)$. This holds because, in the intersection correspondence between strata of ${}^p\calGr_G$ and strata of $\calGr_G$ which have non-empty intersection with this embedded copy, dimensions are constant modulo $2$.
\end{proof}

Thus, we may compose $M$ with the lifting functor $LL$ induced by the Smith functor in Section \ref{liftingsection} and obtain a functor
\begin{align*}
P:\mathrm{Par}^0_{sph}(\calGr_G;\bO)\lra \mathrm{Par}^0_{sph}({}^p\calGr_G;\ff).
\end{align*}
Since the right-hand category is $\ff$-linear, Proposition \ref{Prop: base change} allows us to factor $P$ through a functor
\begin{align*}
\underline{\Psm}:\mathrm{Par}^0_{sph}(\calGr_G;\ff)\lra \mathrm{Par}^0_{sph}({}^p\calGr_G;\ff)\cong\mathrm{Par}^0_{sph}(\calGr_G;\ff).
\end{align*}
On objects, $\underline{\Psm}$ lifts to the corresponding (normal) parity complex with $\bO$-coefficients (see Proposition \ref{Prop: mod p}), then applies the functor $P$. We then identify the component ${}^p\calGr_G$ with $\calGr_G$ to view $\underline{\Psm}$ as an endofunctor.

We now construct an extension of $\underline{\Psm}$ to the full category $\Perv_{sph}(\calGr_G;\ff)$. Using Treumann's modular Tate-category \cite[Section 4]{Tr}, we consider the the modular version of Smith functor 
$$
\Psm_{0}:D^b_{\vp,sph}(\calGr_G,\ff)\lra \mathrm{Sh}_{sph}({}^p\calGr_G, \mathcal{T}_0).
$$
As in the integral case of Section \ref{Section: Smith}, this is given by $\ast$-restricting to ${}^p\calGr_G$, and then localizing. Consider the functor 
\[
 \epsilon_\ast:\mathrm{Sh}_{sph}({}^p\calGr_G, \mathcal{T}_0)\lra D_{sph}({}^p\calGr_G, \ff),
\]
given, on the level of bounded complexes, by tensoring with the Tate complex over $\ff[\vp]$ and taking $\vp$-invariants. The notation is chosen to be consistent with the derived invariants functor from Section \ref{sec:const}.

Despite the unboundedness of the resulting complex, we claim that it is reasonable to take the $0^{th}$ perverse cohomology of the output. Indeed, such complexes have finite dimensional support and the $0^{th}$ perverse cohomology functor factors through a truncation functor $\tau_{\leq -n}\tau_{\geq n}$ for $n>0$ sufficiently large.
Thus, the modular Smith category admits a functor \emph{perverse Tate cohomology}
$$
^pH^0_{Tate}:={}^pH^0\circ  \epsilon_\ast:\mathrm{Sh}_{sph}({}^p\calGr_G, \mathcal{T}_0)\lra \Perv_{sph}({}^p\calGr_G,\ff)\cong\Perv_{sph}(\calGr_G,\ff)
$$
 by taking the $0^{th}$ perverse cohomology of the unbounded complex. Here again, we identify sheaves on ${}^p\calGr_G$ with those on $\calGr_G.$
 

\begin{Prop}\label{Prop: smith extension}
The composition 
$$
{}^pH^0_{Tate}\circ \Psm_{0}:\Perv_{sph}(\calGr_G,\ff)\cong \Perv_{\mathbb{C}^*,sph}(\calGr_G,\ff)\lra\Perv_{sph}(\calGr_G,\ff)
$$
 gives an endofunctor extending the functor
\begin{align*}
\underline{\Psm}:\mathrm{Par}^0_{sph}(\calGr_G;\ff)\lra \mathrm{Par}^0_{sph}(\calGr_G;\ff)
\end{align*}
\end{Prop}

\begin{proof}
By the construction, $\underline{\Psm}$ sits in a commutative diagram
\begin{equation*}
 \begin{tikzcd}
\mathrm{Par}^0_{sph}(\calGr_G;\bO)\ar[r," M"]\ar[d,"\ff"] &\mathrm{Par}^0_{\vp,sph}(\calGr_G;\bO)\ar[r,"\Psm"]&\mathrm{Par}^0_{sph}(\calGr_G;\T_\infty)\ar[d,"L"]\\
\mathrm{Par}^0_{sph}(\calGr_G;\ff)\ar[rr,"\underline{\Psm}"]&&\mathrm{Par}^0_{sph}(\calGr_G;\ff).
\end{tikzcd}
\end{equation*}
Additionally, since the modular reduction functor preserves the thick subcategories of perfect objects, we obtain a functor 
\[
\ff_{Tate}:\mathrm{Sh}_{sph}(\calGr_G;\T_\infty)\lra \mathrm{Sh}_{sph}(\calGr_G;\T_0)
\]
sitting in the commutative diagram
\[
 \begin{tikzcd}
D_{sph}^b(\calGr_G;\bO)\ar[r,"\bT"]\ar[d,"\ff"] &\mathrm{Sh}_{sph}(\calGr_G;\T_\infty)\ar[d,"\ff_{Tate}"]\\
D_{sph}^b(\calGr_G;\ff)\ar[r,"\bT_0"]&\mathrm{Sh}_{sph}(\calGr_G;\ff),
\end{tikzcd}
\]
where $\bT_0$ is the modular Tate extension of scalars functor from \cite{Tr}. Thus, we have a commutative diagram
\begin{equation*}
 \begin{tikzcd}
\mathrm{Par}^0_{sph}(\calGr_G;\bO)\ar[r,"\Psm\circ M"]\ar[d,"\ff"] &\mathrm{Par}^0_{sph}(\calGr_G;\T_\infty)\ar[d,"L"]\ar[dr,"\ff_{Tate}"]\\
\mathrm{Par}^0_{sph}(\calGr_G;\ff)\ar[r,"\underline{\Psm}"]&\mathrm{Par}^0_{sph}(\calGr_G;\ff)\ar[r,"\bT_0"]&\mathrm{Sh}_{sph}(\calGr_G;\ff).
\end{tikzcd}
\end{equation*}
For any $\bO$-parity sheaf $\F$, it follows that $L\Psm M(\F) = \underline{\Psm}(\ff(\F))$ is perverse. In particular, $ \epsilon_\ast\ff_{Tate}\Psm M(\F)\in D(\calGr_G;\ff)$ is a direct sum of even shifts of perverse sheaves. 

Now let $\F$ be an object of $\mathrm{Par}^0_{sph}(\calGr_G;\ff)$; we may realize $\F$ as the cocone of a map $\G\xrightarrow{p} \G$, where $\G$ is an object in $\mathrm{Par}^0_{sph}(\calGr_G;\bO)$ and the map is given by multiplication by $p$. Then $ \epsilon_\ast\Psm_0(\F)$ is the cocone of 
\[
 \epsilon_\ast\ff_{Tate}\Psm M(\G)\xrightarrow{p(=0)}  \epsilon_\ast\ff_{Tate}\Psm M(\G)
\]
in $D_{sph}(\calGr_G;\ff),$ where the map is now the zero map. Therefore,
\[
 \epsilon_\ast\Psm_0(\F)= \epsilon_\ast\ff_{Tate}\Psm M(\G)\oplus  \epsilon_\ast\ff_{Tate}\Psm M(\G)[1]
\]
is a periodic sum of all the shifts of $\underline{\Psm}\ff(\G)=\underline{\Psm}(\F).$ In particular, its degree zero perverse cohomology is just the perverse sheaf $\underline{\Psm}(\F).$
\end{proof}
With this proposition in mind, we adopt the notation $\underline{\Psm}:={}^pH^0_{Tate}\circ \Psm_{0}$.

\subsection{Relation to Frobenius contraction}\label{Section: Frobenius}

We now prove Theorem \ref{Thm: intro motivation}, which we restate for the convenience of the reader.
\begin{Thm}\label{Thm: main app}
The functor
\begin{align*}
\underline{\Psm}:\mathrm{Perv}_{sph}(\calGr_G;\ff)\to \mathrm{Perv}_{sph}(\calGr_G;\ff).
\end{align*}
corresponds to the Frobenius contraction functor under the geometric Satake equivalence.
\end{Thm}
\begin{proof}
This is accomplished in two steps: first we work on the additive subcategory of spherical parity sheaves and then give an extension to the full category of spherical perverse sheaves.

Recalling that $\F r^{(-1)}$ restricts to an endofunctor on $\mathrm{Tilt}(G^\vee_{\ff})$, we need to show that the diagram
\[
    \begin{tikzcd}
\mathrm{Par}^0_{sph}(\calGr_G;\ff) \ar[d] \ar[r,"\underline{\Psm}"]&\mathrm{Par}^0_{sph}(\calGr_G;\ff)\ar[d]\\ 
\mathrm{Tilt}(G^\vee_{\ff})  \ar[r,"\F r^{(-1)}"]&\mathrm{Tilt}(G^\vee_{\ff})   
\end{tikzcd}
\]
commutes, where the vertical arrows are the equivalence induced by the geometric Satake equivalence. Indeed, this gives the top face of the cube 
\[
 \begin{tikzcd}[row sep=1.5em, column sep = 1.5em]
  \mathrm{Par}^0_{sph}(\calGr_G;\ff)   \arrow[rr] \arrow[dr, swap] \arrow[dd,swap] &&
    \mathrm{Par}^0_{sph}(\calGr_G;\ff)\arrow[dd] \arrow[dr] \\
    & \mathrm{Tilt}(G^\vee_{\ff}) \arrow[rr] \ar[dd]&&
   \mathrm{Tilt}(G^\vee_{\ff}) \arrow[dd] \\
  \mathrm{Par}^n_{sph}(\calGr_T;\ff)  \arrow[rr] \arrow[dr] && \mathrm{Par}^n_{sph}(\calGr_T;\ff) \arrow[dr] \\
   &\mathrm{Tilt}(T^\vee_{\ff})   \arrow[rr] &&\mathrm{Tilt}(T^\vee_{\ff}),  
    \end{tikzcd}
\]
the bottom face of which is the diagram
\[
    \begin{tikzcd}
\mathrm{Par}^n_{sph}(\calGr_T;\ff) \ar[d] \ar[r,"\underline{\Psm}"]&\mathrm{Par}^n_{sph}(\calGr_T;\ff)\ar[d]\\ 
\mathrm{Tilt}(T^\vee_{\ff})  \ar[r,"\F r^{(-1)}"]&\mathrm{Tilt}(T^\vee_{\ff}) ,  
\end{tikzcd}
\]
whose vertical arrows simply forget the grading. Here $n$ denotes the normalization of indecomposable parity complexes on $\calGr_T$ inherited from $\calGr_G$, so that 
\begin{align*}
\mathrm{Par}^n_{sph}(\calGr_T;\ff)=\bigoplus_{\lambda\in\mathbb{X}_\bullet(T)}\mathrm{Vect}_{\ff}[\langle 2\rho,\lambda\rangle ].
\end{align*}

The action of $\vp$ is trivial on $Gr_T$, so that the functor $\underline{\Psm}$ sends $V_\lambda[\langle 2\rho,\lambda\rangle ]$ to $V_{\lambda/p}[\langle 2\rho,\lambda/p\rangle ]$ if $p$ divides $\lambda$, and $0$ otherwise. Thus, the bottom face is commutative. The two side faces of the cube are also commutative, as they are just (\ref{diagram: restriction}) stating that the geometric Satake equivalence intertwines the functors $Res^{G^\vee}_{T^\vee}$ and $R^G_T$. The back face of the cube is commutative since $\underline{\Psm}$ commutes with hyperbolic localization with respect to a $\mathbb{G}_m$-action which commutes with the $\vp$ action; this is an easy exercise. 

The front face of the cube commutes by the very definition of $\F r^{(-1)}$. The two vertical maps on this front face are both the restriction functor $Res^{G^\vee}_{T^\vee}$, which is faithful and injective on objects (since its source category is the category of tilting modules, which are determined by their characters). It is a formal consequence that the top face is commutative, as required.

We now deduce the result for the extension
\[
\underline{\Psm}={}^PH_{Tate}\circ\Psm_0:\Perv_{sph}(\calGr_G,\ff)\lra\Perv_{sph}(\calGr_G,\ff).
\] 
For the moment, let $C:\mathrm{Rep}(G^\vee_{\ff})\lra \mathrm{Rep}(G^\vee_{\ff})$ denote the functor corresponding to $\underline{\Psm}$.

Fixing a regular dominant cocharacter $\mu$ of $T$, recall that hyperbolic localization with respect to $\mu$ commutes with functor $\Psm_0$. Using the t-exactness property of hyperbolic localization in this context (\cite{MV}), we see that hyperbolic localization intertwines $^pH^0_{Tate}$ with the functor
$$
 D_{sph}^b(\calGr_T, \T_0)\lra D_{sph}^b(\calGr_T,\ff)
$$
which over the point $\lambda$ of $\calGr_T$ sends a complex to its degree $\langle-2\rho,\lambda\rangle$ Tate cohomology vector space, placed in homological degree $\langle-2\rho,\lambda\rangle$.

The upshot is that, applying the geometric Satake equivalence, the induced endofunctor $C$ of $\mathrm{Rep}(G^\vee_{\ff})$ satisfies the two following conditions:
\begin{enumerate}
    \item On tilting modules, $C$ is the Frobenius contraction functor;
    \item On underlying representations of $T^\vee_{\ff}$, $C$ is the Frobenius contraction functor.
\end{enumerate}
In particular, point 2 shows that $C$ is exact. But an exact endofunctor of $\mathrm{Rep}(G^\vee_{\ff})$ which takes tilting modules to tilting modules is determined by what it does to tilting modules. Indeed, in general given an abelian category $A$, a full additive subcategory $B$ and an exact endofunctor $C'$ of $A$ which sends $B$ to itself, the natural functor
$$
K^b(B)\lra K^b(A)\lra D^b(A)
$$
fits into a commutative square
\[
    \begin{tikzcd}
   K^b(B)\ar[r]\ar[d,swap,"KC'|_{B}"] &D^b(A)\ar[d,"DC'"]\\
  K^b(B)\ar[r]&D^b(A).
    \end{tikzcd}
\]
Applying this in the case 
\[
A=\mathrm{Rep}(G^\vee_{\ff}),\quad B=Tilt(G^\vee_{\ff}), \quad C=C',
\]
the functor $K^b(B)\lra D^b(A)$ is an equivalence. Then, $C$ is determined by $DC$, which is determined by $KC|_B$, which is determined by $C|_B$ as required.

\end{proof}

\bibliography{bib}
\bibliographystyle{crelle}



\end{document}